\documentclass[onecolumn,referee]{svjour3}       
\smartqed
\usepackage[latin1]{inputenc}
\usepackage{amsmath}
\usepackage{amssymb}
\usepackage{url}
\usepackage{hyperref}%
\usepackage{graphicx}
\usepackage{ifthen}
\usepackage{dsfont}
\usepackage{slashbox,subfig,verbatim}
\usepackage{xargs}
\usepackage[table]{xcolor}
\reversemarginpar
\usepackage[novbox]{pdfsync}           
\usepackage[disable]{todonotes}
\usepackage{algorithm,algpseudocode}

\usepackage{aliascnt}
\usepackage{hyperref}

\def\nset{\mathbb{N}}
\def\rset{\mathbb{R}}

\newcommand{\NN}{\mathbb{N}}
\newcommand{\PP}{\mathbb{P}}

\newcommand{\RR}{\mathbb{R}}

\def \PP{\mathbb{P}}
\def \RR{\mathbb{R}}

\newcommand{\bx}{\mathbf{x}}

\newcommand{\bA}{\mathbf{A}}
\newcommand{\bL}{\mathcal{M}}
\newcommand{\RPfield}{\mathcal{B}(\RR^+)}

\newcommand{\eqdef}{\ensuremath{\stackrel{\mathrm{def}}{=}}}
\newcommand{\eqnum}[1]{\ensuremath{\stackrel{(#1)}{=}}}
\newcommand{\eqsp}{\;}
\newcommand{\one}{\mathbf{1}}
\newcommand{\esp}[1]{\mathbb{E}\left[#1\right]}
\newcommand{\cesp}[2]{\mathbb{E}\left[ #1 \middle| #2 \right]}
\newcommand{\var}[1]{\mathbb{V}\mathrm{ar}\left(#1\right)}
\newcommand{\rmd}{\mathrm{d}}

\newcommand{\epart}[2][]{\ifthenelse{\equal{#1}{}}{\boldsymbol{X}}{X}_{#2}\ifthenelse{\equal{#1}{}}{}{^{#1}}}
\newcommand{\espart}[2][]{\ifthenelse{\equal{#1}{}}{\boldsymbol{\hat{X}}}{\xi}_{#2}\ifthenelse{\equal{#1}{}}{^{N_1}}{^{N_1,#1}}}
\newcommand{\Xset}[1]{\mathbb{E}_{#1}}
\newcommand{\Xfield}[1]{\mathcal{E}_{#1}}
\newcommand{\trans}[1]{M_{#1}}
\newcommand{\omeas}{m}
\newcommand{\semi}[1][]{Q\ifthenelse{\equal{#1}{}}{}{_{#1}}}
\newcommand{\semiNorm}[1][]{P\ifthenelse{\equal{#1}{}}{}{_{#1}}}
\newcommand{\banach}[1]{\mathcal{B}_b(\Xset{#1})}
\newcommand{\measSet}[1]{\mathcal{P}(\Xset{#1})}
\newcommand{\kac}[1][]{\eta\ifthenelse{\equal{#1}{}}{}{_{#1}}}
\newcommand{\unkac}[1][]{\gamma\ifthenelse{\equal{#1}{}}{}{_{#1}}}
\newcommand{\Pkac}[1][]{\eta^{N_1}\ifthenelse{\equal{#1}{}}{}{_{#1}}}
\newcommand{\Punkac}[1][]{\gamma^{N_1}\ifthenelse{\equal{#1}{}}{}{_{#1}}}
\newcommand{\pot}[1][]{g\ifthenelse{\equal{#1}{}}{}{_{#1}}}
\newcommand{\func}[1][]{f\ifthenelse{\equal{#1}{}}{}{_{#1}}}
\newcommand{\PPsi}[1]{\Psi_{#1}}
\newcommand{\Peps}[1]{\epsilon_{#1}}
\newcommand{\eweight}[2][]{\ifthenelse{\equal{#1}{}}{\boldsymbol{\omega}}{\omega}_{#2}^{\ifthenelse{\equal{#1}{}}{}{#1}}}

\newcommand{\Nepart}[2][]{\ifthenelse{\equal{#1}{}}{\boldsymbol{\Xi}}{\boldsymbol{X}}_{#2}\ifthenelse{\equal{#1}{}}{}{^{#1}}}
\newcommand{\Neweight}[2][]{\ifthenelse{\equal{#1}{}}{\boldsymbol{\Omega}}{\boldsymbol{\omega}}_{#2}\ifthenelse{\equal{#1}{}}{}{^{#1}}}

\newcommand{\Nsmallpart}[2]{X_{#2}^{#1}}
\newcommand{\Nsmallweight}[2]{\omega_{#2}^{#1}}
\newcommand{\Ntpart}[2][]{\pmb{\widetilde{X}}_{#2}\ifthenelse{\equal{#1}{}}{}{^{#1}}}
\newcommand{\Nspart}[2][]{\ifthenelse{\equal{#1}{}}{\boldsymbol{\widehat{\Xi}}}{\widehat{\boldsymbol{X}}}_{#2}\ifthenelse{\equal{#1}{}}{}{^{#1}}}
\newcommand{\NXset}[1]{\pmb{\mathbb{E}}_{#1}}
\newcommand{\NXfield}[1]{\boldsymbol{\mathcal{E}}_{#1}}
\newcommand{\Ntrans}[1]{\boldsymbol{M}_{#1}}
\newcommandx{\Nomeas}[1][1=N_2]{\boldsymbol{m}^{#1}}
\newcommand{\Nsemi}[1][]{\boldsymbol{Q}\ifthenelse{\equal{#1}{}}{}{_{#1}}}
\newcommand{\NsemiNorm}[1][]{\boldsymbol{P}\ifthenelse{\equal{#1}{}}{}{_{#1}}}
\newcommand{\Nbanach}[1]{\mathcal{B}_b(\NXset{#1})}
\newcommand{\NmeasSet}[1]{\mathcal{P}(\NXset{#1})}
\newcommand{\Nkac}[1][]{\boldsymbol{\eta}\ifthenelse{\equal{#1}{}}{}{_{#1}}}
\newcommand{\Nunkac}[1][]{\pmb{\boldsymbol\gamma}\ifthenelse{\equal{#1}{}}{}{_{#1}}}
\newcommand{\NPkac}[1][]{\boldsymbol{\eta}^{N_2}\ifthenelse{\equal{#1}{}}{}{_{#1}}}
\newcommand{\NPunkac}[1][]{\pmb{\boldsymbol\gamma}^{N_2}\ifthenelse{\equal{#1}{}}{}{_{#1}}}
\newcommand{\PNkac}[1][]{\boldsymbol{\eta}^{N_2}\ifthenelse{\equal{#1}{}}{}{_{#1}}}
\newcommand{\PNunkac}[1][]{\pmb{\boldsymbol\gamma}^{N_2}\ifthenelse{\equal{#1}{}}{}{_{#1}}}
\newcommand{\PNkactilde}[1][]{\boldsymbol{\widetilde{\eta}}^{N_2}\ifthenelse{\equal{#1}{}}{}{_{#1}}}
\newcommand{\Npot}[1][]{\boldsymbol{g}\ifthenelse{\equal{#1}{}}{}{_{#1}}}
\newcommand{\Nfunc}[1][]{\boldsymbol{f}\ifthenelse{\equal{#1}{}}{}{_{#1}}}
\newcommandx{\nlsemigrp}[1][1=]{\Phi_{#1}}
\newcommand{\NPPsi}[1]{\boldsymbol{\Psi}_{#1}}

\title{On parallel implementation of Sequential Monte Carlo methods: the island particle model}
\titlerunning{Island models}
\authorrunning{C.~Vergé and al.}
\author{Christelle Vergé \and Cyrille Dubarry \and Pierre Del Moral  \and Eric Moulines}
\institute{Pierre Del Moral \at Centre INRIA Bordeaux Sud Ouest - 351 Cours de la Libération, 33405 Talence Cedex, \email{pierre.del-moral@inria.fr} \and Cyrille Dubarry \at SAMOVAR, CNRS UMR 5157 - Institut T\'el\'ecom/T\'el\'ecom SudParis,  9 rue Charles Fourier, 91000 Evry \and Eric Moulines \at LTCI, CNRS UMR 8151 - Institut T\'el\'ecom/T\'el\'ecom ParisTech, 46 rue Barrault,  75634 Paris Cedex 13, France, \email{eric.moulines@telecom-paristech.fr} \and Christelle Vergé \at ONERA - The French Aerospace Lab, F-91761 Palaiseau, \at CNES - 18 avenue Edouard Belin, 31401 Toulouse Cedex 9, \email{christelle.verge@onera.fr}}
\thanks{This work is supported by the Agence Nationale de la Recherche through the 2009-2012 project Big MC}
\begin{document}
\maketitle

\begin{abstract}
The approximation of the Feynman-Kac semigroups by systems of interacting particles is a very active research field, with applications in many different areas. In this paper, we study the parallelization of such approximations. The total population of particles is divided into sub-populations, referred to as \emph{islands}. The particles within each island follow the usual selection / mutation dynamics.
We show that the evolution of each island is also driven by a Feynman-Kac semigroup, whose transition and potential can be explicitly related to ones of the original problem.
Therefore, the same genetic type approximation of the Feynman-Kac semi-group may be used at the island level; each island might undergo  selection / mutation algorithm. We investigate the impact of the population size within each island and the number of islands, and study different type of interactions.
We find conditions under which introducing interactions between islands is beneficial. The theoretical results are supported by some Monte Carlo experiments.
\end{abstract}

\section{Introduction}
Numerical approximation of Feynman-Kac semigroups by systems of interacting particles is a very active  field of researchs. Interacting particle systems  are increasingly used to sample complex high dimensional distributions in a wide range of applications including
nonlinear filtering, data assimilation problems, rare event sampling,
hidden Markov chain parameter estimation, stochastic control problems, financial
mathematics; see for example \cite{doucet:defreitas:gordon:2001}, \cite{chopin:2002}, \cite{delmoral:2004}, \cite{cappe:moulines:2005},  \cite{delmoral:hu:wu:2012} and the references therein.

Let $(\Xset{n},\Xfield{n})_{n \geq 0}$ be a sequence of measurable spaces. Denote by $\banach{n}$ the Banach space of all bounded and measurable real valued functions $f$ on $\Xset{n}$, equipped with the uniform norm. Let $(\pot[n])_{n\in \nset}$  be a sequence of measurable \emph{potential functions}, $\pot[n]:\Xset{n}\rightarrow \rset^+$.
Let $(\Omega,\mathcal{F},\PP)$ be a probability space. In the sequel, all the processes are defined on this probability space. Let $(X_{n})_{n \in \nset}$ be a non-homogenous Markov chain on the sequence of state-spaces $(\Xset{n})_{n \in \nset}$ with initial distribution $\kac_{0}$ on $(\Xset{0},\Xfield{0})$ and Markov kernels $(\trans{n})_{n\in \nset^*}$ \footnote{a Markov kernel on $\Xset{n} \times \Xfield{n+1}$ is a function $\trans{n+1}:\Xset{n} \times \Xfield{n+1} \rightarrow [0;1]$, such that, for all $x_n \in \Xset{n}$, $A_{n+1}\mapsto \trans{n+1}(x_n,A_{n+1})$ is a probability measure on $(\Xset{n+1},\Xfield{n+1})$  and for any $A_{n+1} \in \Xfield{n+1}$, $x_n \mapsto \trans{n+1}(x_n,A_{n+1})$ is a measurable function.}.
We associate to the sequences of potential functions $(\pot[n])_{n\in \nset}$ and Markov kernels $(\trans{n})_{n \in \nset^*}$ the sequence of \textit{Feynman-Kac measures}, defined for all $n \geq 1$ and for any $\func[n]\in \banach{n}$ by
\begin{align}\label{eq:defEta}
&\kac_n(\func[n])\eqdef \unkac_n(\func[n])/\unkac_n(1) \eqsp, \\
&\unkac_n(\func[n])\eqdef \esp{\func[n](X_n)~\prod_{0\leq p<n}\pot[p](X_p)} \\
& \quad \quad \quad \quad =\int \unkac_0(\rmd x_0)\left[\prod_{0\leq p < n} \pot[p](x_p) \trans{p+1}(x_p, \rmd x_{p+1}) \right] \func[n](x_n) \eqsp, \label{eq:defGamma}
\end{align}
where we have set by convention $\kac[0](\func[0]) = \unkac[0](\func[0]) \eqdef \esp{\func[0](X_0)}$.

The sequences of distributions $(\kac_n)_{n \geq 0}$ and $(\unkac_n)_{n \geq 0}$ are approximated sequentially using interacting particle systems (IPS). Such particle approximations are often referred to as sequential Monte Carlo (SMC) methods. The IPS consists in approximating for each $n \in \nset$ the probability $\kac_n$  by a set of $N_1$ \emph{particles} $(\epart[i]{n})_{i=1}^{N_1}$ which are generated recursively.
Typically, the update of the particles may be decomposed into a mutation and a selection step. For example, the bootstrap algorithm proceeds as follows. In the \emph{selection step} the particles are first sampled with weights proportional to the potential functions. In the \emph{mutation step}, a new generation of particles $(\epart[i]{n+1})_{i=1}^{N_1}$ is generated from the selected particles using the kernel $\trans{n+1}$. The asymptotic behavior of such particle approximation is now well understood (see \cite{delmoral:2004} and \cite{delmoral:hu:wu:2012}).

Feynman-Kac measures appear naturally in the filtering problem for Hidden Markov Model (HMM). Recall that a HMM is a pair of discrete time random processes $(X,Y)=(X_n,Y_n)_{n \in \nset}$, where $(X_n)_{n \geq 0}$ is the hidden state process (often called signal) and $(Y_n)_{n \geq 0}$ are the observations. To fix the ideas, $X_n$ and $Y_n$ take values in $\mathbb{X} \subset \RR^k$ and $\mathbb{Y} \subset \RR^l$. The state sequence is assumed to be a Markov chain with transition probability density $m(x,x')$ and initial density $m_0$ (both with respect to some common dominating measure $\mu$). In this case, for all $n \geq 0$, $\Xset{n} = \mathbb{X}$ and for all $A \in \mathcal{B}(\mathbb{X})$, $\trans{n}(x,A)=\int_{A}{m(x,x') \mu(\rmd x')}$, where $\mathcal{B}(\mathbb{X})$ is the Borel $\sigma$-field. The observations $(Y_n)_{n \geq 0}$ are conditionally independent given $X$ and for all $n \in \NN^{*}$, $Y_n$ has a conditional density $g(X_n,.)$ with respect to a reference measure $\nu$ such that
$\PP(Y_n \in B | X_n)=\int_{B}{g(X_n,y) \nu(dy)}$, for all $B \in \mathcal{B}(\mathbb{Y})$.
Here the potential functions are the likelihood of the observations $g_n(x)=g(x,Y_n)$.
In such settings, $\gamma_n$ is the joint distribution of $X_n$ and $Y_0,...,Y_{n-1}$, $\eta_n$ is the predictive distribution of $X_n$ conditionally on $Y_0,...,Y_{n-1}$, and $\unkac_n(1)$ is the likelihood of the sequence of observations $Y_0,...,Y_{n-1}$.

Particle filtering is computationally an intensive method. Parallel computations provides an appealing solution to tackle this issue (see \cite{durham:geweke:2011} and the references therein for an in-depth description of parallelization of Bayesian computations).
The basic idea to implement interacting particle system in parallel goes as follows: instead of considering a single large batch of $N= N_1 N_2$ particles, the population is divided into $N_2$ batches of $N_1$ particles. These batches are referred in the sequel to as \textit{islands}. The terminology \textit{island} is borrowed from dynamic populations theory (like the genetic type interacting particle model).
The particles within each island are selected and mutates, as described above. We might also introduce interactions among \emph{islands}.

In this paper we introduce the \emph{island particle models}. As we will see below, we may cast the island particle model in the Feynman-Kac framework, with appropriately defined potentials and transition kernels.
The key observation is that the marginal distribution of the island Feynman-Kac model w.r.t. any individual coincide with \eqref{eq:defGamma}.
This interpretation allows to use the interacting particle model at the island level.

The study of the island particle model gives rise to several challenging theoretical questions. In this paper,
we investigate the impact of the number of particles in each island $N_1$ compared to the number of islands $N_2$ for a given total number of particles $N\eqdef N_1 N_2$, for the double bootstrap algorithm, where the bootstrap mechanism is used both within and between the islands.
We focus on the asymptotic bias and variance when both $N_1$ and $N_2$ goes to infinity.
Fluctuation theorem and non-asymptotic results will be present in a forthcoming paper.
We also investigate when and why introducing interactions at the island level improves the accuracy of the particle approximation. Intuitively, the trade-off might be understood as follows. When the $N_2$ islands are run independently, the bias induced in each island only depends on their population size $N_1$; when $N_1$ is small compared to the total number $N$, the bias will be large (and is of course not reduced by averaging across the islands).
To reduce the bias,  introducing an interaction between the islands is beneficial. However, this interaction increases the variance, due to the selection step.
If we consider the mean squared error, the interaction is beneficial when the improvement associated to the bias correction is not offset by the variance increase. When the number of particles $N_1$ within each island is \emph{small} and the number of islands $N_2$ is \emph{large}, then the interaction is typically beneficial. On the contrary, when $N_2<<N_1$, the interaction between islands may increase the mean squared error. We then propose a method, based on a generalization of the effective sample size, this time computed at the island level, which always achieve a lower mean squared error than the independent island model.

The paper is organized as follows. In \autoref{sec:AlgorithmDerivation} the interacting particle approximation of the Feynman-Kac model is first reviewed. The island Feynman-Kac model is then introduced.
We first investigate the  double bootstrap algorithm, in which selection and mutation are applied at each iteration within and across the islands.
The asymptotic bias and variance of this algorithm is presented in \autoref{sec:asymptotic}.  The Feynman-Kac interpretation of the island model leads to several  interacting island algorithms, based on different approximations of Feynman-Kac flows. Some of these are introduced and analyzed in \autoref{sec:extensions}. Some numerical experiments are reported to support our findings and illustrate the impact of the numbers of islands and particles within each island in \autoref{sec:NumericalSimulations}.

\section{Algorithm derivation}
\label{sec:AlgorithmDerivation}
\label{subsec:SingleIsland}

In this section, we introduce the island particle model. We first briefly recall the bootstrap approximation of Feynman-Kac measures.

According to the definitions \eqref{eq:defEta} and \eqref{eq:defGamma} of the sequences of the Feynman-Kac measures
$(\kac_n)_{n \in \nset}$ and $(\unkac_n)_{n \in \nset}$, for all $\func[n+1] \in \banach{n+1}$ we get
$$
\unkac_{n+1}(\func[n+1])= \kac_{n+1}(\func[n+1]) \unkac_{n+1}(1),
$$
and since,
$$
\unkac_{n+1}(1)= \unkac_{n}(\pot[n])=\kac_{n}(\pot[n])\unkac_{n}(1),
$$
an easy induction shows that
$$
\unkac_{n+1}(1)=\prod_{0 \leq p < n+1} \kac_p(\pot[p])
$$
and then,
\begin{equation}\label{eq:decompositionEta}
\unkac_{n+1}(\func[n+1]) = \kac_{n+1}(\func[n+1])~\prod_{0\leq p<n+1}\kac_p(\pot[p])\eqsp.
\end{equation}
Moreover, the sequence $(\kac_n)_{n \in \nset}$ satisfy a nonlinear recursive relation. Indeed,
\begin{equation}
\label{eq:NonlinearRecursion}
\kac_{n+1}(\func[n+1])= \frac{ \unkac_n(\pot[n] \trans{n+1} \func[n+1])}{ \unkac_n(\pot[n] \trans{n+1} 1)}=
\frac{ \kac_n(\pot[n] \trans{n+1} \func[n+1])}{ \kac_n(\pot[n])}.
\end{equation}
Let $\measSet{n}$ be the set of probability measures on $\Xset{n}$.
Using the Boltzmann-Gibbs transformation $\PPsi{n}:\measSet{n} \to \measSet{n}$, defined for all $\mu_{n} \in \measSet{n}$ by
\begin{equation}\label{eq:defPsi}
\PPsi{n}(\mu_{n})(\rmd x_n) \eqdef \dfrac{\pot[n](x_n)~\mu_n(\rmd x_n)}{\mu_n(\pot[n])} \eqsp,
\end{equation}
the recursion \eqref{eq:NonlinearRecursion} may be rewritten as
\begin{equation} \label{eq:transportPhi}
\kac_{n+1}=\PPsi{n}(\kac_{n})\trans{n+1}\eqsp.
\end{equation}

The sequence of probability $(\kac_n)_{n \in \nset}$ can be approximated using the bootstrap algorithm. Other approximations can also be considered as well, but we only introduce the bootstrap for notational simplicity.
Let $N_1$ be a positive integer. For any nonnegative integer $n$ we denote by
\begin{equation}
\label{eq:defNXset}
(\NXset{n},\NXfield{n})\eqdef (\Xset{n}^{N_1},\Xfield{n}^{\otimes N_1}) \eqsp,
\end{equation}
the product space (the dependence of $\NXset{n}$ and $\NXfield{n}$ in $N_1$ is implicit). Thereafter, we omit to write the $\sigma$-field $\NXfield{n}$ when there will be no confusion.
 We define the Markov kernel $\Ntrans{n+1}(\bx_n,\rmd \bx_{n+1})$ from $\NXset{n}$ into $\NXset{n+1}$ as follows: for any $\bx_n = (x_n^1,\dots,x_n^{N_1})\in \NXset{n}$, we set
\begin{eqnarray}\label{eq:defNtrans}
\Ntrans{n+1}(\bx_n,\rmd \bx_{n+1}) &\eqdef & \prod_{1\leq i\leq N_1}~\sum_{j=1}^{N_1}\dfrac{\pot[n](x_n^j)}{\sum_{k=1}^{N_1} \pot[n](x_n^k)}\trans{n+1}(x_n^j,\rmd x^i_{n+1})\eqsp.
\end{eqnarray}
In other words, this transition can be interpreted as follows:
 \begin{itemize}
 \item In the \emph{selection step}, the components of the vector $\bx_n$ are selected with probabilities proportional to their potential $\{ \pot[n](x_n^i) \}_{i=1}^{N_1}$;
 \item In the \emph{mutation step}, the selected coordinates move conditionally independently to new positions using  the Markov kernel $\trans{n+1}$.
 \end{itemize}
Let us introduce the particles and their evolution.
Define by $(\epart{n})_{n \geq 0}$ the Markov chain where for each $n \in \nset$,
\begin{equation}\label{eq:defSingleParticles}
\epart{n}=(\epart[1]{n},\dots,\epart[N_1]{n}) \in \NXset{n},
\end{equation}
with initial distribution $\Nkac[0] \eqdef \kac[0]^{\otimes N_1}$ and transition kernel $\Ntrans{n+1}$.
Denote by $\omeas^{N_1}$ the empirical measure on $\NXset{n}$, defined as the kernel on $\NXset{n} \times \Xset{n}$ by
$$
 \omeas^{N_1}(\bx_n, \rmd z_{n})\eqdef\frac{1}{N_1}\sum_{i=1}^{N_1} \delta_{x_n^i}(\rmd z_n)\eqsp,
$$
where $\delta_{x_n}$ is the dirac mass at $x_n \in \Xset{n}$.
Equation \eqref{eq:decompositionEta} suggests the following $N_1$-particle approximations of the measures $\kac_n$ and $\unkac_n$ respectively defined for $\func[n] \in \banach{n}$ by
\begin{align}
\label{eq:approxDef_1}
&\Pkac_n(\func[n])\eqdef\omeas^{N_1}\func[n](\epart{n})=\frac{1}{N_1} \sum_{i=1}^{N_1} \func[n](\epart{n}^i) \\
\label{eq:approxDef_2}
&\Punkac_n(\func[n])\eqdef \Pkac_n(\func[n])~\prod_{0\leq p<n}\Pkac_p(\pot[p]) = \Pkac_n(\func[n])~\Punkac_n(1) \eqsp.
\end{align}

For $\bx_n =(x_n^1, \cdots, x_n^{N_1}) \in \NXset{n}$, define the potential function
\begin{equation}
\label{eq:defG}
\Npot[n](\bx_n)\eqdef \omeas^{N_1}\pot[n](\bx_n) = \frac{1}{N_1} \sum_{i=1}^{N_1}\pot[n](x_n^i)\eqsp.
\end{equation}
The sequences of transition kernels $(\Ntrans{n})_{n \in \nset}$  and potential functions $(\Npot[n])_{n \in \nset}$ given by \eqref{eq:defNtrans} and \eqref{eq:defG}, respectively, define the Feynman-Kac process.
The associated sequences of Feynman-Kac measures are defined, for each $\Nfunc[n] \in \Nbanach{n}$, by the following recursions
\begin{align}
&\Nkac_0(\Nfunc[0]) \eqdef \Nunkac_0(\Nfunc[0])= \esp{\Nfunc[0](\epart{0})} \eqsp, \label{eq:defNGamma} \\
&\Nkac_n(\Nfunc[n])\eqdef \Nunkac_n(\Nfunc[n])/\Nunkac_n(1), \quad \text{for all} \ n \geq 1, \label{eq:defNEta} \\
&\Nunkac_n(\Nfunc[n]) \eqdef \esp{\Nfunc[n](\epart{n})~\prod_{0\leq p<n}\Npot[p](\epart{p})} \eqsp, \quad \text{for all} \ n \geq 1. \label{eq:interpretation-feynman-kac-Nflow}
\end{align}
where $(\epart{n})_{n \geq 0}$ is a Markov chain with initial distribution $\Nunkac_0$ and transition kernel $\Ntrans{n}$.
The key result, justifying the introduction of the island particle models, is the following theorem which links $(\kac_n,\unkac_n)_{n \geq 0}$ and $(\Nkac_n,\Nunkac_{n})_{n \geq 0}$.

\begin{theorem}
\label{theo:unbiased}
For any $\Nfunc[n] \in \Nbanach{n}$ of the form $\Nfunc[n](\bx_n) = N_1^{-1} \sum_{i=1}^{N_1} \func[n](x_n^i) $ where $\func[n] \in \banach{n}$,
\begin{equation}
\Nunkac_n(\Nfunc[n]) = \unkac_n(\func[n]) \quad \mbox{and} \quad \Nkac_n(\Nfunc[n]) = \kac_n(\func[n])\eqsp.
\end{equation}
\end{theorem}
\begin{proof}
See \autoref{subsec:ProofUnbiased}.
\end{proof}

The Feynman-Kac model $(\Nunkac_n, \Nkac_n)_{n \geq 0}$ can be approximated by an interacting particle system \emph{at the island level}. We first describe the \emph{double bootstrap} algorithm where the bootstrap is also applied across the islands (this algorithm shares some similarities with \cite{chopin:jacob:papaspiliopoulos:2012}). This is only one of the many possible algorithms that can be derived from this interpretation of the Feynman-Kac model at the island level; see \autoref{sec:extensions} for other approximations.

Define by $\NmeasSet{n}$ the set of probabilities measures on $\NXset{n}$.
One can easily check that the sequence of measures $(\Nkac_n)_{n \geq 0}$ satisfies the following recursion
\begin{equation} \label{eq:NtransportPhi}
\Nkac_{n+1}=\NPPsi{n}(\Nkac_{n})\Ntrans{n+1}\eqsp,
\end{equation}
where $\NPPsi{n}: \NmeasSet{n} \to \NmeasSet{n}$ is the Boltzmann-Gibbs transformation defined for any $\mu_n \in \NmeasSet{n}$ by
$$
\NPPsi{n}(\mu_n)(\rmd \bx_n) \eqdef \dfrac{\Npot[n](\bx_n)~\mu_n(\rmd \bx_n)}{\mu_n(\Npot[n])} \eqsp.
$$
Let $N_2$ be a positive integer.
We define the Markov kernel $\bL_{n+1}$ from $(\NXset{n}^{N_2},\NXfield{n}^{\otimes N_2})$ to $(\NXset{n+1}^{N_2},\NXfield{n+1}^{\otimes N_2})$ as follows:
for any $(\bx_n^1,\dots,\bx_n^{N_2})\in \NXset{n}^{N_2}$ and $(\bx_{n+1}^1,\dots,\bx_{n+1}^{N_2})\in \NXset{n+1}^{N_2}$, we put
\begin{multline}\label{eq:defBigNtrans}
\bL_{n+1}((\bx_n^1,\dots,\bx_n^{N_2}),\rmd(\bx_{n+1}^1,\dots,\bx_{n+1}^{N_2})) \\
\eqdef \prod_{1\leq i\leq N_2} \sum_{j=1}^{N_2} \dfrac{\Npot[n](\bx_n^j)}{\sum_{k=1}^{N_2} \Npot[n](\bx_n^k)}\Ntrans{n+1}(\bx_n^j,\rmd\bx^i_{n+1})\eqsp.
\end{multline}

For each $n \in \nset$, $(\Nepart[1]{n},\dots,\Nepart[N_2]{n}) \in \NXset{n}^{N_2}$ is a population of $N_2$ interacting islands each with $N_1$ individuals. The process $\{(\Nepart[1]{n},\dots,\Nepart[N_2]{n})\}_{n \geq 0}$\ is a Markov chain with the transition kernel $(\bL_{n+1})_{ n \geq 0}$.

In this interpretation, the $N_2$-particle model defined above can be seen as an interacting particle approximation of the island Feynman-Kac measures $\{(\Nkac_{n},\Nunkac_{n})\}_{n \geq 0}$.
The transition $\bL_{n+1}$ can be interpreted as follows:
\begin{itemize}
\item In the \textit{selection step}, we sample randomly $N_2$ islands among the current islands $\left(\Nepart[i]{n}\right)_{1\leq i\leq N_2}\in  \NXset{n}^{N_2}$
with probability proportional to the empirical mean of the potentials in each island $\Npot[n](\Nepart[i]{n}) = N_1^{-1} \sum_{j=1}^{N_1} \pot[n](\Nsmallpart{i,j}{n})$, $1 \leq i \leq N_2$.
\item In the \textit{mutation transition},  the selected islands are independently updated using the Markov transition $\Ntrans{n+1}$.
\end{itemize}
Also observe that for $N_1=1$, every island has a single particle. In this situation, the island Feynman-Kac model  coincides with the $N_2$-particle model associated with the Feynman-Kac measures $\kac_n$.

Denote by $\Nomeas$ the empirical measure defined for any $\Nfunc[n] \in \Nbanach{n}$  and $(\bx_n^1,\dots,\bx_n^{N_2}) \in \NXset{n}^{N_2}$ by
$$
 \Nomeas \Nfunc[n] (\bx_n^1,\dots,\bx_n^{N_2})\eqdef\frac{1}{N_2}\sum_{i=1}^{N_2} \Nfunc[n](\bx_n^i)\eqsp.
$$

The $N_2$-particle approximations of the measures $\Nkac_n$ and $\Nunkac_n$ are defined for any $\Nfunc[n] \in \Nbanach{n}$ by
\begin{align}\label{eq:NapproxDef}
&\NPkac_n(\Nfunc[n])\eqdef\Nomeas \Nfunc[n] (\Nepart[1]{n},\dots,\Nepart[N_2]{n})\eqsp, \\
&\NPunkac_n(\Nfunc[n])\eqdef \NPkac_n(\Nfunc[n])~\prod_{0\leq p<n}\NPkac_p(\Npot[p]) = \NPkac_n(\Nfunc[n])~\NPunkac_{n}(1) \eqsp.
\end{align}

\begin{algorithm}[h]
    \begin{algorithmic}[1]
        \caption{\quad Bootstrap within bootstrap island filter}\label{alg:BootstrapWithinBootstrap}
        \State \underline{Initialization:}
        \For{$i$ from $1$ to $N_2$}
            \State Sample $N_1$ independent random variables $\Nepart[i]{0}=\left(\Nsmallpart{i,j}{0}\right)_{j=1}^{N_1}$ from $\kac_0$.
        \EndFor
        \For{$p$ from $0$ to $n-1$}
            \State \underline{Selection step between islands:}
            \State Sample $\boldsymbol{I}_p = (I_p^i)_{i=1}^{N_2}$ multinomially
            with probability proportional to $\left(\frac{1}{N_1} \sum_{j=1}^{N_1} \pot[p](\Nsmallpart{i,j}{p})\right)_{i=1}^{N_2}$.
            \State \underline{Island mutation step:}
            \For{$i$ from $1$ to $N_2$}
                \State \underline{Particle selection within each island:}
                \State Sample $\boldsymbol{J}_p^i = (J_p^{i,j})_{j=1}^{N_1}$ multinomially
            with probability proportional to $\left(\pot[p](\Nsmallpart{I_p^i,j}{p})\right)_{j=1}^{N_1}$.
                \State \underline{Particle mutation:}
                \State For $1 \leq j \leq N_1$, sample conditionally independently $\Nsmallpart{i,j}{p+1}$ from the Markov kernel $\trans{p+1}(\Nsmallpart{I_p^i,L_p^{i,j}}{p},\cdot)$,
                where $L_p^{i,j}=J_p^{I_p^i,j}$.
            \EndFor
        \EndFor
        \State Approximate $\kac[n](\func[n])$ by $\displaystyle \dfrac{1}{N_1 N_2} \sum_{i=1}^{N_2} \sum_{j=1}^{N_1} \func[n]\left(\Nsmallpart{i,j}{n}\right)$.
    \end{algorithmic}
\end{algorithm}

\section{Asymptotic analysis of the double bootstrap algorithm}
\label{sec:asymptotic}
The bootstrap particle approximation of the Feynman-Kac semigroup can be studied using the techniques introduced in \cite{delmoral:2004} and further developed in \cite{delmoral:hu:wu:2012}. For $\ell \in \nset$, consider the finite kernel $\semi[\ell+1]$ from $(\Xset{\ell},\Xfield{\ell})$ into $(\Xset{\ell+1},\Xfield{\ell+1})$ given for all $x_{\ell} \in \Xset{\ell}$ by
\begin{equation}
\label{eq:defQ}
\semi[\ell+1](x_{\ell},\rmd x_{\ell+1})\eqdef \pot[\ell](x_{\ell}) \trans{\ell+1}(x_{\ell},\rmd x_{\ell+1}) \eqsp.
\end{equation}
For $p < n$, define by $\semi[p,n]$ the finite kernel from $(\Xset{p},\Xfield{p})$ into $(\Xset{n},\Xfield{n})$ as the following product
\begin{equation}
\label{eq:defQbold}
\semi[p,n]\eqdef \semi[p+1]\semi[p+2]\dots \semi[n] \eqsp,
\end{equation}
and set by convention $\semi[n,n] \eqdef \mathrm{I}_n$ where $\mathrm{I}_n$ is the identity kernel on $(\Xset{n},\Xfield{n})$. With this definition, the linear semigroup associated with the sequence of unnormalized Feynman-Kac measures $(\unkac_n)_{n \in \NN}$ may be equivalently expressed as follows
\begin{equation}
\label{eq:linear-feynman-kac-flow}
\unkac_n=\unkac_p \semi[p,n] \eqsp.
\end{equation}
For any $x_p \in \Xset{p}$, $A_n \in \Xfield{n}$, $\semi[p,n]$ may be written as the following conditional expectation,
$$
\semi[p,n](x_p,A_n)=\cesp{\mathds{1}_{A_n}(X_n)~\prod_{p\leq q<n}\pot[q](X_q)}{X_p=x_p} \eqsp,
$$
where $(X_{n})_{n\geq 0}$ is the non-homogenous Markov chain on the sequence of state-spaces $(\Xset{n},\Xfield{n})_{n\geq 0}$ with initial distribution $\kac_{0}$ and Markov kernels $(\trans{n})_{n\geq 1}$.

According to \eqref{eq:defEta}, $\kac_n= \unkac_n/ \unkac_n(1)$ implies that $\kac_n = \unkac_p \semi[p,n]/\unkac_p\semi[p,n](1)$. Denote by $\nlsemigrp[n+1]$ the mapping from $\measSet{n}$ to $\measSet{n+1}$ given, for any $\mu_n \in \measSet{n}$ by
\begin{equation}
\label{eq:definition-nlsemigroup-island}
\nlsemigrp[n+1](\mu_n) \eqdef \PPsi{n}(\mu_n) \trans{n+1} = \dfrac{\mu_n \semi[n+1] }{\mu_n \semi[n+1](1)} \eqsp.
\end{equation}
Since $\kac_p = \unkac_p /\unkac_p(1)$, these relations may be equivalently rewritten as
\begin{equation}\label{eq:defQbar}
\kac_n = \frac{\kac_p \semi[p,n]}{\kac_p\semi[p,n](1)} = \nlsemigrp[p,n](\kac_p)\eqsp, \\
\end{equation}
where $\nlsemigrp[p,n]= \nlsemigrp[n] \circ \nlsemigrp[n-1] \circ \dots \circ \nlsemigrp[p+1]$
is the nonlinear semigroup associated to the normalized Feynman-Kac measures $(\eta_n)_{n \geq 0}$. This nonlinear semigroup may be associated to the \emph{potential kernels}
\begin{equation}
\label{eq:definition-potential-kernel}
\semiNorm_{p,n} \eqdef \dfrac{\semi_{p,n}}{\kac_p\semi_{p,n}(1)} = \dfrac{\unkac_p(1)}{\unkac_n(1)} \semi_{p,n}\eqsp,
\end{equation}
and therefore
\begin{equation}
\label{eq:linear-normalized-feynman-kac-flow}
\kac_n=\kac_p\semiNorm_{p,n} \eqsp.
\end{equation}
For $\ell \in \nset$, consider the finite kernel $\Nsemi[\ell+1]$ from $(\NXset{\ell},\NXfield{\ell})$ into $(\NXset{\ell+1},\NXfield{\ell+1})$ for any $\bx_{\ell} \in \NXset{\ell}$ by
$$
\Nsemi[\ell+1](\bx_{\ell},\rmd \bx_{\ell+1})\eqdef \Npot[\ell](\bx_{\ell}) \Ntrans{\ell+1}(\bx_{\ell},\rmd \bx_{\ell+1})\eqsp,
$$
where $\Ntrans{\ell}$ is defined in \eqref{eq:defNtrans} and $\Npot[\ell]$ in \eqref{eq:defG}.
For $p \leq n$, define by $\Nsemi[p,n]$ the finite kernel from $(\NXset{p},\NXfield{p})$ into $(\NXset{n},\NXfield{n})$ by the equation $\Nsemi[p,n]\eqdef \Nsemi[p+1]\Nsemi[p+2]\dots \Nsemi[n] \eqsp.$
Note that, for any $\bx_p \in \NXset{p}$, $\bA_n \in \NXfield{n}$,
$$
\Nsemi[p,n](\bx_p,A_n)=\cesp{\mathds{1}_{\bA_n}(\epart{n})~\prod_{p\leq q<n}\Npot[q](\epart{q})}{\epart{p}=\bx_p} \eqsp,
$$
where $(\epart{n})_{n \geq 0}$ is the island Markov chain defined in \eqref{eq:defSingleParticles}.
With this notation, we may rewrite \eqref{eq:defNGamma} as
$
\Nunkac_n=\Nunkac_p \Nsemi[p,n] \eqsp.
$
According to \eqref{eq:defNEta}, $\Nkac_n = \Nunkac_n / \Nunkac_n(1)$ implies that $\Nkac_n= \Nunkac_p \Nsemi[p,n]/\Nunkac_p \Nsemi[p,n](1)$, and then
$$
\Nkac_n = \frac{\Nkac_p \Nsemi[p,n]}{\Nkac_p \Nsemi[p,n](1)} = \Nkac_p \NsemiNorm[p,n]\eqsp,
$$
where $\NsemiNorm_{p,n}$ are given by
\[
\NsemiNorm_{p,n} \eqdef \dfrac{\Nsemi_{p,n}}{\Nkac_p \Nsemi_{p,n}(1)} = \dfrac{\Nunkac_p(1)}{\Nunkac_n(1)} \Nsemi_{p,n} \eqsp.
\]
According to \autoref{theo:unbiased}, $\Nunkac_p(1)= \unkac_p(1)$ and $\Nunkac_n(1)= \unkac_n(1)$, which implies that
\[
\NsemiNorm_{p,n}= \dfrac{\unkac_p(1)}{\unkac_n(1)} \Nsemi_{p,n}\eqsp.
\]
To analyse the fluctuation of the interacting particle approximation $(\Pkac_n)_{n \geq 0}$ around their limiting values $(\kac_n)_{n \geq 0}$, we introduced th \emph{local sampling errors}. We first decompose the difference $\Punkac_n - \unkac_n$ as follows
\begin{equation}
\label{eq:diff-unkac}
\Punkac_n - \unkac_n = \sum_{p=1}^n \left[ \Punkac_p \semi_{p,n} - \Punkac_{p-1} \semi_{p-1,n} \right] + \Punkac_0 \semi_{0,n} - \unkac_n \eqsp.
\end{equation}
For any $p \geq 1$, note that
\begin{multline*}
\Punkac_{p-1} \semi_p= \Punkac_{p-1}(1)~\Pkac_{p-1} \semi_p= \Punkac_{p-1}(1)~\Pkac_{p-1}(\pot[p-1])~\nlsemigrp[p](\Pkac_{p-1}) \\
= \Punkac_{p-1}(1)~\dfrac{\Punkac_{p-1}(\pot[p-1])}{\Punkac_{p-1}(1)}~\nlsemigrp[p](\Pkac_{p-1})= \Punkac_{p-1}(\pot[p-1])~\nlsemigrp[p](\Pkac_{p-1}) = \Punkac_p(1)~\nlsemigrp[p](\Pkac_{p-1}) \eqsp.
\end{multline*}
Plugging in this relation in the local error yields to
\begin{multline*}
\Punkac_p \semi_{p,n} - \Punkac_{p-1} \semi_{p-1,n} = \Punkac_p \semi_{p,n} - \Punkac_{p-1} \semi_p \semi_{p,n}\\
= \left(\Punkac_p  - \Punkac_p(1) \nlsemigrp[p](\Pkac_{p-1}) \right)  \semi_{p,n}= \Punkac_p(1)\left(\Pkac_p  - \nlsemigrp[p](\Pkac_{p-1})\right) \semi_{p,n} \eqsp,
\end{multline*}
which, together with \eqref{eq:diff-unkac}, imply that,
\begin{equation}\label{eq:defWgamma}
W_n^{\unkac,N_1} \eqdef \sqrt{N_1}\left[\Punkac_n-\unkac_n\right] = \sum_{p=0}^n \Punkac_p(1) W_p^{N_1} \semi_{p,n} \eqsp,
\end{equation}
where the  local errors $(W_p^{N_1})_{p \geq 0}$ are defined  by
\begin{equation}
\label{eq:defW}
W_0^{N_1}=\sqrt{N_1}(\Pkac_0-\kac_0) \quad \text{and} \quad W_p^{N_1} = \sqrt{N_1} \left[ \Pkac_p - \nlsemigrp[p](\Pkac_{p-1}) \right],~~\text{for~all}~~ p \geq 1 \eqsp.
\end{equation}
The following results, adapted from \cite[Corollary 9.3.1, pp. 295-298]{delmoral:2004}, establishes the convergence of $(W_p^{N_1})_{1 \leq p \leq n}$ to centered Gaussian fields.
\begin{theorem}
\label{theo:localAsympto}
For the bootstrap filter, for any fixed time horizon $n \geq 1$, the sequence $(W_p^{N_1})_{1 \leq p \leq n}$ converges in law, as $N_1$ goes to infinity, to a sequence of $n$ independent centered Gaussian random fields $(W_p)_{0 \leq p \leq n}$ with variance given, for any bounded function $\func[p] \in \banach{p}$, and $1 \leq p \leq n$, by
\begin{equation}
\label{eq:definition-moment-Wp}
\esp{W_p(\func[p])^2} = \kac_p\left[ \left(\func[p]-\kac_p\func[p]\right)^2\right]
\eqsp.
\end{equation}
\end{theorem}
Now, consider the sequence of random fields $(W_n^{\kac,N_1})_{n \geq 0}$ defined for any function $\func[n] \in \banach{n}$ by
\begin{align}
\label{eq:defWeta}
W_n^{\kac,N_1}(\func[n]) \eqdef \sqrt{N_1}\left[\Pkac_n-\kac_n\right](\func[n]) & = \sqrt{N_1}\Pkac_n[\func[n]-\kac_n(\func[n])] \\
& =\sqrt{N_1}\dfrac{\Punkac_n(\func[n]-\kac_n(\func[n]))}{\Punkac_n(1)} \eqsp.
\end{align}
Using the fact that $\unkac_n(f_n-\kac_n(f_n))=0$ and \eqref{eq:defWgamma}, we may write
\begin{align}
W_n^{\kac,N_1}(\func[n])=\sqrt{N_1}\dfrac{(\Punkac_n-\unkac_n)(\func[n]-\kac_n(\func[n]))}{\Punkac_n(1)} =  \dfrac{W_n^{\unkac,N_1}\left(\func[n]-\kac_n(\func[n])\right)}{\Punkac_n(1)} \eqsp.
\end{align}
The decomposition \eqref{eq:defWgamma} and \eqref{eq:defWeta}, combined with the Slutsky's lemma, imply the following asymptotic decomposition (which remains valid for more general algorithms than the bootstrap algorithm)
\begin{theorem}
\label{theo:localAsympto-1}
Assume that the sequence of local errors $(W_p^{N_1})_{1 \leq p \leq n}$ converges in law, as $N_1$ goes to infinity, to a sequence of $n$ independent centered Gaussian random fields $(W_p)_{1 \leq p \leq n}$.
Then, the sequence of random fields $(W_n^{\unkac,N_1})_{N_1 \geq 0}$ converges in law, as $N_1$ goes to infinity, to the Gaussian random fields $W_n^{\unkac}$ defined for any bounded function $\func[n]$ in $\banach{n}$ by
\begin{equation}\label{eq:defWGammaLim}
W_n^{\unkac}(\func[n]) \eqdef \sum_{p=0}^n \unkac_p(1) W_p(\semi_{p,n}\func[n])  = \unkac_n(1) \sum_{p=0}^n  W_p(\semiNorm_{p,n}\func[n])\eqsp,
\end{equation}
where $\semiNorm_{p,n}$ is defined in \eqref{eq:definition-potential-kernel}. The sequence of random fields $(W_n^{\kac,N_1})_{N_1 \geq 0}$ converges in law, as $N_1$ goes to infinity, to the Gaussian random fields $W_n^{\kac}$ defined for any function $\func[n] \in \banach{n}$ by
\begin{equation}\label{eq:defWetaLim}
W_n^{\kac}(\func[n]) \eqdef \sum_{p=0}^n W_p(\semiNorm_{p,n}(\func[n]-\kac_n(\func[n]))) \eqsp.
\end{equation}
\end{theorem}

The asymptotic bias and variance for the single island interacting particle approximation of the sequence of Feynman-Kac measure formulated in the forthcoming theorem result almost immediately from \autoref{theo:localAsympto}.
\begin{theorem}
\label{theo:singleIslandAsymptotics}
Assume that the sequence of local errors $(W_p^{N_1})_{1 \leq p \leq n}$ converges in law, as $N_1$ goes to infinity, to a sequence of $n$ independent centered Gaussian random fields $(W_p)_{1 \leq p \leq n}$.
Then, for any time horizon $n\geq 0$ and any bounded function $\func[n] \in \banach{n}$, we have
\begin{align*}
&\lim_{N_1 \rightarrow \infty} N_1 \esp{\Pkac_n(\func[n]) - \kac_n(\func[n])} = B_n(\func[n]) \eqsp,\\
&\lim_{N_1 \rightarrow \infty} N_1 \var{\Pkac_n(\func[n])} = V_n(\func[n]) \eqsp,
\end{align*}
with
\begin{align}
\label{eq:defBn}
&B_n(\func[n]) \eqdef -\sum_{p=0}^n \esp{W_p(\semiNorm_{p,n}(1))W_p(\semiNorm_{p,n}(\func[n]-\kac_n(\func[n])))} \eqsp, \\
\label{eq:defVn}
&V_n(\func[n]) \eqdef \sum_{p=0}^n \esp{\{W_p(\semiNorm_{p,n}(\func[n]-\kac_n(\func[n])))\}^2}\eqsp.
\end{align}
\end{theorem}
When the bootstrap algorithm is applied, we get the following expressions for $B_n(\func[n])$ and $V_n(\func[n])$ using \autoref{theo:localAsympto}:
\label{coro:singleIslandAsymptoticsBootstrap}
\begin{align}
\label{eq:defBn-bootstrap}
&B_n(\func[n]) = -\sum_{p=0}^n \kac_p \left(\semiNorm_{p,n}(1)\semiNorm_{p,n}(\func[n]-\kac_n(\func[n]))\right) \eqsp, \\
\label{eq:defVn-bootstrap}
&V_n(\func[n]) = \sum_{p=0}^n \kac_p \left(\semiNorm_{p,n}(\func[n]-\kac_n(\func[n]))^2 \right)\eqsp.
\end{align}
\begin{proof}
See \autoref{subsec:ProofSingleIsland}.
\end{proof}

We now compute the bias and the variance for the double bootstrap algorithm. The asymptotic behavior of the bias and the variance is derived in the following theorem using techniques adapted from \cite{delmoral:2004}.
\begin{theorem}
\label{theo:InteractingAsymptotic}
For the double bootstrap algorithm, for any time horizon $n\geq 0$ and any $\func[n] \in \banach{n}$, we have
\begin{align*}
&\lim_{N_1 \rightarrow \infty}\lim_{N_2 \rightarrow \infty} N_1N_2 \esp{\PNkac_n(\omeas^{N_1}\func[n]) - \kac_n(\func[n])} = B_n(\func[n])+\widetilde{B}_n(\func[n]) \eqsp,\\
&\lim_{N_1 \rightarrow \infty}\lim_{N_2 \rightarrow \infty} N_1N_2 \var{\PNkac_n(\omeas^{N_1}\func[n])} = V_n(\func[n])+\widetilde{V}_n(\func[n]) \eqsp,
\end{align*}
where $B_n(\func[n])$ and $V_n(\func[n])$ are defined respectively in \eqref{eq:defBn} and in \eqref{eq:defVn},
and where $\widetilde{B}_n(\func[n])$ and $\widetilde{V}_n(\func[n])$ are given by:
\begin{align}
&\widetilde{B}_n(\func[n]) \eqdef  -\sum_{\ell=0}^{n} (n-\ell) \esp{W_\ell(\semiNorm_{\ell,n}(1)) W_\ell(\semiNorm_{\ell,n}(\func[n]-\kac_n(\func[n])))} \label{eq:defBtilden}\\
&  \quad \quad \quad + \sum_{\ell=0}^{n}  \esp{W_\ell\left(\sum_{p=\ell}^n\semiNorm_{\ell,p}(1)\right) W_\ell(\semiNorm_{\ell,n}(\func[n]-\kac_n(\func[n])))} \eqsp,  \notag \\
&\widetilde{V}_n(\func[n]) \eqdef \sum_{\ell=0}^{n} (n-\ell) \esp{W_\ell(\semiNorm_{\ell,n}(\func[n]-\kac_n(\func[n])))^2} \eqsp. \label{eq:defVtilden}
\end{align}
\end{theorem}
When the bootstrap algorithm is applied, we get the following expressions for $\widetilde{B}_n(\func[n])$ and $\widetilde{V}_n(\func[n])$ using \autoref{theo:localAsympto}:
\label{coro:singleIslandAsymptoticsBootstrap2}
\begin{align}
\label{eq:defBntilde-bootstrap}
&\widetilde{B}_n(\func[n]) = -\sum_{\ell=0}^{n} (n-\ell)~\kac_{\ell}\left((\semiNorm_{\ell,n}(1)-\kac_n(1))  \semiNorm_{\ell,n}(\func[n]-\kac_n(\func[n])) \right) \\
&  \quad \quad \quad + \sum_{\ell=0}^{n}  \kac_\ell \left( \left(\sum_{p=\ell}^n(\semiNorm_{\ell,p}(1)-\kac_p(1))\right) \semiNorm_{\ell,n}(\func[n]-\kac_n(\func[n])))\right) \eqsp,  \notag \\
\label{eq:defVntilde-bootstrap}
&\widetilde{V}_n(\func[n]) = \sum_{\ell=0}^{n} (n-\ell)~\kac_{\ell} \left(\semiNorm_{\ell,n}(\func[n]-\kac_n(\func[n]))^2 \right)\eqsp.
\end{align}
\begin{proof}
See \autoref{subsec:ProofInteractingIslands}.
\end{proof}


We can also consider the case where the $N_2$ islands are kept independent (a bootstrap filter is still applied within each island, but there is no interaction between islands). To that purpose, denote by $(\Ntpart[i]{n})_{i=1}^{N_2}$ $N_2$ independent islands of size $N_1$, each evolving using the bootstrap filter and, define the estimator of $\kac_n(\func[n])$ for any $\Nfunc[n] \in \Nbanach{n}$, given by the empirical mean across islands
$$\PNkactilde_n(\Nfunc[n]) \eqdef \frac{1}{N_2} \sum_{i=1}^{N_2} \Nfunc[n] (\Ntpart[i]{n}) \eqsp.$$
For functions $\Nfunc[n]$ on $\NXset{n}$ of the form $\Nfunc[n](\bx_{n})=\omeas^{N_1}\func[n](\bx_{n})$, with $\func[n]\in\banach{n}$, we have
$$\PNkactilde_n(\Nfunc[n]) = \frac{1}{N_2} \sum_{i=1}^{N_2} \omeas^{N_1}\func[n](\Ntpart[i]{n}) = \frac{1}{N_1N_2} \sum_{i=1}^{N_2} \sum_{j=1}^{N_1} \func[n](\Ntpart[i,j]{n})\eqsp.$$
The asymptotic behavior of the bias and variance of $\omeas^{N_1}\func[n](\Ntpart[i]{n})$ may be easily deduced from the one of $\Pkac_n(\func[n])$; \autoref{theo:singleIslandAsymptotics} implies that
\begin{theorem}
\label{theo:IndependentAsymptotic}
For any time horizon $n\geq 0$ and any $\func[n] \in \banach{n}$, we have
\begin{align*}
&\lim_{N_1 \rightarrow \infty} N_1 \left\{\esp{\PNkactilde_n(\omeas^{N_1}\func[n])} - \kac_n(\func[n])\right\} = B_n(\func[n]) \eqsp,\\
&\lim_{N_1 \rightarrow \infty} N_1N_2 \var{\PNkactilde_n(\omeas^{N_1}\func[n])} = V_n(\func[n]) \eqsp,
\end{align*}
where $B_n(\func[n])$ and $V_n(\func[n])$ are defined respectively in \eqref{eq:defBn} and \eqref{eq:defVn}.
\end{theorem}
The variance of the particle approximation is inversely proportional to $N_1N_2$, but because the islands do not interact, the bias is independent of $N_2$ and is inversely proportional to $N_1$.

As shown by \autoref{theo:InteractingAsymptotic} and \autoref{theo:IndependentAsymptotic}, a trade-off has to be made between the bias and the variance to decide which of the two estimators $\PNkac_n$ and $\PNkactilde_n$ is the best. We can compare the mean squared error (MSE) when the islands interact or when they are kept independent. The MSE for independent islands is given by $\frac{V_n(\func[n])}{N_1N_2} + \frac{B_n(\func[n])^2}{N_1^2}$ whereas the MSE for the double bootstrap is given by $\frac{V_n(\func[n])+\widetilde{V}_n(\func[n])}{N_1N_2}$. Therefore,
$$ \frac{V_n(\func[n])+\widetilde{V}_n(\func[n])}{N_1N_2} < \frac{V_n(\func[n])}{N_1N_2} + \frac{B_n(\func[n])^2}{N_1^2}
 \quad \Leftrightarrow \quad N_1 < \frac{B_n(\func[n])^2}{\widetilde{V}_n(\func[n])} N_2 \eqsp.$$
Consequently, the double bootstrap algorithm outperforms the independent islands when the number of particles $N_1$ within each island is small compared to the number of islands $N_2$; the interaction improves the bias (which is independent of $N_2$ when the islands are kept independent). On the contrary, when $N_1$ is larger than $N_2$, the variance increase introduced by the interaction (because of the selection step) may be larger than the bias reduction.

\section{Extensions}
\label{sec:extensions}
In \autoref{sec:asymptotic} we have described and analyzed an interacting island model where the bootstrap algorithm is used both within and across the islands. Of course, other IPS approximations may be considered within and across islands. We will describe how the results of the previous sections may be adapted.
The IPS approximation of each individual island may be cast in the Feynman-Kac framework.
This section is devoted to check these conditions for various IPS approximations.

	\subsection{Epsilon-bootstrap interaction}
$\epsilon$-bootstrap interaction is a variant of the bootstrap, in which the selection step is slightly modified: only a fraction of the particles are resampled.
Let $ \Peps{n}$ be a nonnegative constant such that $ \Peps{n}~\left\|\pot[n]\right\|_{\infty}\in [0,1]$, where $\left\|\pot[n]\right\|_{\infty}=\sup_{x_n \in \Xset{n}}|\pot[n](x_n)|$.
For any measure $\mu_n \in \measSet{n}$, define $S_{n,\mu_n}$ the Markov kernel on $(\Xset{n},\Xfield{n})$ given for $x_n \in \Xset{n}$ and $A_n \in \Xfield{n}$ by
\begin{equation}\label{eq:defS}
S_{n,\mu_n}(x_n,A_n) \eqdef \Peps{n}~\pot[n](x_n) \delta_{x_n}(A_n)+ \left(1- \Peps{n}~\pot[n](x_n)\right)\PPsi{n}(\mu_{n})(A_n)\eqsp,
\end{equation}
where $\PPsi{n}$ is defined in \eqref{eq:defPsi}.
$\epsilon$-bootstrap interaction algorithm proceeds as follows. At iteration $n$, a particle $\epart[i]{n}$ is kept with a probability equal to $ \Peps{n}~\pot[n](\epart[i]{n})$ or resampled with a probability $ 1-\Peps{n}~\pot[n](\epart[i]{n})$. Resampling a particle consists in replacing it by a particle selected at random in the current population with weights proportional to their potential $(\pot[n](\epart[1]{n}),\dots, \pot[n](\epart[N_1]{n}))$. Then, each selected particle is independently updated according to the Markov kernel $\trans{n+1}$. When $\Peps{n}=0$, all the particles are resampled, which correspond to the bootstrap filter.
Define the Markov kernel $\Ntrans{n+1}(\bx_n,d\bx_{n+1})$ from $\NXset{n}$ into $\NXset{n+1}$ by
\begin{equation}\label{eq:defNtransEpsilon}
\Ntrans{n+1}(\bx_n,\rmd \bx_{n+1}) \eqdef \prod_{1\leq i\leq N_1}~S_{n,\Pkac_n} \trans{n+1} (x_n^i,\rmd x_{n+1}^i) \eqsp.
\end{equation}
Consider a Markov chain $(\epart{n})_{n \geq 0}$ where for each $n \in \nset$,
$\epart{n}=(\epart[1]{n},\dots,\epart[N_1]{n}) \in \NXset{n}$,
with initial distribution $\kac[0]$ and transition kernel $\Ntrans{n+1}$.
Define the same approximations of the measures $\kac_n$ and $\unkac_n$ as in \eqref{eq:approxDef_1} and \eqref{eq:approxDef_2}.
Then, consider the island Feynman-Kac model associated to the Markov chain \eqref{eq:defNGamma} and the potential function \eqref{eq:defG}. The associated sequence $\{(\Nkac[n],\Nunkac[n])\}_{n \geq 0}$ of Feynman-Kac measures is given for all $\Nfunc[n] \in \Nbanach{n}$ by
\begin{align}
\label{eq:defNkacEps}
&\Nkac_0(\Nfunc[0]) \eqdef \Nunkac_0(\Nfunc[0])= \esp{\Nfunc[0](\epart{0})} \eqsp,  \\
&\Nkac[n](\Nfunc[n])\eqdef \Nunkac[n](\Nfunc[n])/\Nunkac[n](1), \quad \text{for all} \ n \geq 1, \\
&\Nunkac[n](\Nfunc[n]) \eqdef \esp{\Nfunc[n](\epart{n})~\prod_{0\leq p<n}\Npot[p](\epart{p})} \eqsp, \quad \text{for all} \ n \geq 1.
\end{align}
 We may establish the following extension of \autoref{theo:unbiased}. Let $\{(x_p^1,\dots,(x_p^{N_1})\}_{0 \leq p \leq n}$ be a population of particles generated by the $\epsilon$-bootstrap interaction algorithm specified by \eqref{eq:defNtransEpsilon}; then,
\begin{theorem}
\label{ProofUnbiasedEpsilon}
For any $\Nfunc[n] \in \Nbanach{n}$ of the form $\Nfunc[n](\bx_n) = N_1^{-1} \sum_{i=1}^{N_1} \func[n](x_n^i) $ with $\func[n] \in \banach{n}$, we get
\begin{equation}
\Nunkac[n](\Nfunc[n]) = \unkac_n(\func[n]) \quad \mbox{and} \quad \Nkac[n](\Nfunc[n]) = \kac_n(\func[n])\eqsp.
\end{equation}
\end{theorem}
\begin{proof}
See \autoref{subsec:ProofUnbiasedEpsilon}.
\end{proof}
For each $n \in \NN$, let $(\Nepart[1]{n},\dots,\Nepart[N_2]{n}) \in \NXset{n}^{N_2}$ be a population of $N_2$ islands each of $N_1$ individuals. The process $(\Nepart[1]{n},\dots,\Nepart[N_2]{n})$ is a Markov chain evolving according to selection and mutation steps, defined as follows
\begin{itemize}
\item \textit{Selection step}: each island $\Nepart[i]{n}$ is kept with a probability equal to $ \epsilon_{n}~\Npot[n](\Nepart[i]{n})$ or resampled with a probability $ 1-\Peps{n}~\Npot[n](\Nepart[i]{n})$. Resampling an island consists in replacing it by an island selected at random in the current population with weights proportional to their potential $(\Npot[n](\Nepart[1]{n})),\dots, \Npot[n](\Nepart[N_1]{n}))$.
\item \textit{Mutation step:} each selected island is updated independently according to the Markov transition $\Ntrans{n+1}$.
\end{itemize}
These islands particles allow to build the $N_2$-particle approximation of the measures $\Nkac[n]$ and $\Nunkac[n]$, for any $\Nfunc[n] \in \Nbanach{n}$, as
\begin{align*}
&\NPkac_n(\Nfunc[n])\eqdef \frac{1}{N_2}\sum_{i=1}^{N_2} \Nfunc[n](\Nepart[i]{n}) \eqsp, \\
&\NPunkac_n(\Nfunc[n])\eqdef \NPkac_n(\Nfunc[n])~\prod_{0\leq p<n}\NPkac_p(\Npot[p]) = \NPkac_n(\Nfunc[n])~\NPunkac_{n}(1) \eqsp.
\end{align*}
For this selection scheme, the following results, adapted from \cite[Corollary 9.3.1, pp. 295-298]{delmoral:2004}, establishes the convergence of $(W_p^{N_1})_{1 \leq p \leq n}$ to centered Gaussian fields:
\begin{theorem}
\label{theo:localAsymptoEpsilon}
For the $\Peps{n}$-bootstrap filter, for any fixed time horizon $n \geq 1$, the sequence $(W_p^{N_1})_{1 \leq p \leq n}$ defined in \eqref{eq:defW} converges in law, as $N_1$ goes to infinity, to a sequence of $n$ independent centered Gaussian random fields $(W_p)_{0 \leq p \leq n}$ with variance given, for any bounded function $\func[p] \in \banach{p}$, and $1 \leq p \leq n$, by
\begin{equation}
\label{eq:definition-moment-Wp_epsilon}
\esp{W_p(\func[p])^2} = \kac_{p-1} \left[S_{p-1,\kac_{p-1}}\trans{p}\func[p]^2-\left(S_{p-1,\kac_{p-1}} \trans{p}\func[p]\right)^2\right]
\eqsp.
\end{equation}
\end{theorem}
This variance is smaller than the variance of the bootstrap algorithm.
\begin{proposition}
\label{cor:choice-epsilon}
The asymptotic variance of $\Pkac[n]$ is smaller with respect to a non-zero sequence $(\Peps{p})_{0 \leq p \leq n-1}$ introduced in \eqref{eq:defS} than in the bootstrap algorithm.
\end{proposition}
\begin{proof}
The proof is given in \autoref{subsec:ProofEpsilon}.
\end{proof}
For example, for $\Peps{p} = \left( \mathrm{essup}_{\kac[p]}(\pot[p]) \right)^{-1}, \eqsp 0 \leq p \leq n$ the asymptotic variance of $\Pkac_n(\func[n]), \eqsp 0 \leq p \leq n$ is lower than for the bootstrap. We can also adapt it at the island level. For instance, \autoref{alg:IPSEpsilonESS} describes the $\Peps{p} = \left(\displaystyle \max_{1 \leq j \leq N_1} \Npot[p](\Nepart[j]{p}) \right)^{-1}$-bootstrap islands interaction with ESS filter within the islands.
\begin{algorithm}[!h]
    \begin{algorithmic}[1]
        \caption{\quad ESS within $\Peps{p}$-bootstrap interaction for $\Peps{p} = \left(\mathrm{essup}_{\kac_p^{N_1}}(\pot[p]) \right)^{-1}$} \label{alg:IPSEpsilonESS}
        \State \underline{Initialization:}
        \For{$i$ from $1$ to $N_2$}
        \State Set $\Neweight[i]{0} = \left(\Nsmallweight{i,j}{0}\right)_{j=1}^{N_1} = (1,\dots,1)$.
        \State Sample $\Nepart[i]{0}=\left(\Nsmallpart{i,j}{0}\right)_{j=1}^{N_1}$ independently distributed according to $\kac_0$.
        \EndFor
        \For{$p$ from $0$ to $n-1$}
            \State \underline{Island selection step:}
            \For{$i$ from $1$ to $N_2$}
            \begin{itemize}
            \item With probability $\Npot[p](\Nepart[i]{p})~/\displaystyle \max_{1 \leq k \leq N_2} \Npot[p](\Nepart[k]{p})$, set $I_p^i=i$.
            \item With probability  $1-\Npot[p](\Nepart[i]{p})~/\displaystyle \max_{1 \leq k \leq N_2} \Npot[p](\Nepart[k]{p})$, sample $I_p^i$ multinomially with probability proportional to $\{\Npot[p](\Nepart[l]{p})~/ \sum_{k=1}^{N_2} \Npot[p](\Nepart[k]{p})\}_{l=1}^{N_2}$.
            \end{itemize}
            \EndFor
            \State \underline{Island mutation step:}
            \For{$i$ from $1$ to $N_2$}
                \State \underline{Particle selection and weight updating within each island:}
                \State Set $N_1^{\mathrm{eff}} = \left(\sum_{j=1}^{N_1}\omega_p^{I_p^i,j}\pot[p](\Nsmallpart{I_p^i,j}{p})\right)^2/\sum_{j=1}^{N_1}\left(\omega_p^{I_p^i,j}\pot[p](\Nsmallpart{I_p^i,j}{p})\right)^2$.
                \If{$N_1^{\mathrm{eff}} \geq \alpha_\mathrm{Particles} N_1$}
                    \State For $1 \leq j \leq N_1$, set $\omega_{p+1}^{i,j} = \omega_p^{I_p^i,j} \pot[p](\Nsmallpart{I_p^i,j}{p})$.
                    \State Set $\boldsymbol{J}_p^i = (J_p^{i,j})_{j=1}^{N_1} = (1, 2, \dots, N_1)$.
                \Else
                    \State Set $\boldsymbol{\omega}_{p+1}^i = \left( \omega_{p+1}^{i,j} \right)_{j=1}^{N_1} = (1, \dots, 1)$.
                    \State Sample $\boldsymbol{J}_p^i = (J_p^{i,j})_{j=1}^{N_1}$ multinomially
            with probability proportional to $\left(\omega_p^{I_p^i,j}\pot[p](\Nsmallpart{I_p^i,j}{p})\right)_{j=1}^{N_1}$.
                \EndIf
            \State \underline{Particle mutation:}
                \State For $1 \leq j \leq N_1$, sample independently $\Nsmallpart{i,j}{p+1}$ according to $\trans{p+1}(\Nsmallpart{I_p^i,L_p^{i,j}}{p},\cdot)$,
                where $L_p^{i,j}=J_p^{I_p^i,j}$.
            \EndFor
        \EndFor
        \State Approximate $\kac[n](\func[n])$ by $\displaystyle \dfrac{1}{N_2\sum_{j=1}^{N_1} \omega_n^{i,j}} \sum_{i=1}^{N_2} \sum_{j=1}^{N_1} \omega_n^{i,j} \func[n]\left(\Nsmallpart{i,j}{n}\right)$.
    \end{algorithmic}
\end{algorithm}

    \subsection{Effective Sample Size interaction}
	\label{subsec:ESS}
	We describe the particle approximation of the probabilities $(\kac[n])_{n \geq 0}$ using the effective sample size (ESS) method introduced in \cite{liu:chen:1995}; see also \cite{liu:2001}, \cite{delmoral:doucet:jasra:2012} and \cite{douc:moulines:2008}.
The difference with the bootstrap filter stems from the selection step of the current particles which is not performed at each step, but only when the importance weights do not satisfy some appropriately defined criterion. Contrary to the bootstrap filter, we now keep both the particles and the weights.
Denote by $x_n^i$ a particle and $w_n^{i}$ its associated weight, assumed to be nonnegative. For a weighted sample $\{(w_n^{i},x_n^{i})\}_{i=1}^{N_1}$, the criterion
\[
\left(\sum_{i=1}^{N_1}w_n^i\pot[n](x_n^i)\right)^2 / \sum_{i=1}^{N_1}\left(w_n^i\pot[n](x_n^i)\right)^2
\]
is the \emph{effective sample size} (ESS).
The algorithm goes as follows. When the ESS is less than $\alpha N_1$, for some $\alpha \in (0,1)$, the particles are multinomially resampled with probabilities proportional to their weights times their potential functions and the weights are all reset to 1. When the ESS is greater than $\alpha N_1$, then the weights are simply multiplied by the potential function. The selected particles are then updated using the transition kernel $\trans{n+1}$. For any nonnegative integer $p$ we set $(\NXset{p},\NXfield{p})\eqdef ((\Xset{p}\times\RR^+)^{N_1},(\Xfield{p}\otimes\RPfield)^{\otimes N_1})$.
Introduce the following set
$$
\Theta_{n,\alpha} = \left\{ \bx_n = \left[ (x_n^1,w_n^1),\dots,(x_n^{N_1},w_n^{N_1})\right]\in \NXset{n} \middle| \dfrac{\left(\sum_{i=1}^{N_1}w_n^i\pot[n](x_n^i)\right)^2}{\sum_{i=1}^{N_1}\left(w_n^i\pot[n](x_n^i)\right)^2} \geq \alpha N_1 \right\} \eqsp.
$$
Define the Markov kernel $\Ntrans{n+1}$ from $\NXset{n}$ into $\NXset{n+1}$   by
\begin{multline}\label{eq:defNtransESS}
\Ntrans{n+1}(\bx_n,\rmd\bx_{n+1}) \eqdef \one_{\Theta_{n,\alpha}}(\bx_n) \left[ \prod_{1\leq i\leq N_1} \delta_{w_n^i\pot[n](x_n^i)}(\rmd w_{n+1}^i)\trans{n+1}(x_n^i,\rmd x^i_{n+1}) \right]\\
+ \one_{\Theta_{n,\alpha}^\mathrm{c}}(\bx_n) \left[ \prod_{1\leq i\leq N_1} \delta_{1}(w_{n+1}^i)
\sum_{j=1}^{N_1} \dfrac{w_{n}^j \pot[n](x_n^j)}{\sum_{k=1}^{N_1} w_{n}^k \pot[n](x_n^k)} \trans{n+1}(x_n^j,\rmd x^i_{n+1}) \right] \eqsp,
\end{multline}
where $\bx_n = \left[ (x_n^1,w_n^1),\dots,(x_n^{N_1},w_n^{N_1}) \right]\in \NXset{n}$ and $\Theta_{n,\alpha}^\mathrm{c}$ is the complement of $\Theta_{n,\alpha}$.
We define a Markov chain $(\epart{n})_{n \geq 0}$ where for each $n \in \nset$,
\begin{equation}\label{eq:defSingleParticlesESS}
\epart{n}=\left[ (\epart[1]{n},\eweight[1]{n}),\dots,(\epart[N_1]{n},\eweight[N_1]{n}) \right] \in \NXset{n} \eqsp,
\end{equation}
with initial distribution $\Nkac[0] \eqdef (\kac[0] \otimes \delta_1)^{\otimes N_1}$ and transition kernel $\Ntrans{n+1}$.
Equation \eqref{eq:decompositionEta} suggests the following $N_1$-particle approximations of the measures $\kac_n$ and $\unkac_n$ defined for $\func[n] \in \banach{n}$ by
\begin{align}\label{eq:approxDefESS}
&\Pkac_n(\func[n]) \eqdef \frac{1}{\sum_{i=1}^{N_1} \eweight[i]{n}}\sum_{i=1}^{N_1} \eweight[i]{n}\func[n]\left(\epart[i]{n}\right) = \omeas^{N_1}\func[n](\epart{n})  \eqsp,\\
\label{eq:approxDefESS-gamma}
&\Punkac_n(\func[n])\eqdef \Pkac_n(\func[n])~\prod_{0\leq p<n}\Pkac_p(\pot[p])=\Pkac_n(\func[n])~\Punkac_{n}(1)\eqsp,
\end{align}
where $\omeas^{N_1}$ stands for the operator given for any $\func[n] \in \banach{n}$ by
$$
\omeas^{N_1}\func[n](\bx_n)\eqdef\frac{1}{\sum_{i=1}^{N_1} w_n^i}\sum_{i=1}^{N_1} w_n^i\func[n](x_n^i)\eqsp.
$$
For $\bx_n =((x_n^1,w_n^1), \dots, (x_n^{N_1},w_n^{N_1})) \in \NXset{n}$, define the potential function
\begin{equation}\label{eq:defGESS}
\Npot[n](\bx_n)\eqdef \omeas^{N_1}\pot[n](\bx_n) = \frac{1}{\sum_{i=1}^{N_1} w_n^i}\sum_{i=1}^{N_1} w_n^i\pot[n]\left(x_n^i\right)\eqsp.
\end{equation}
We consider the island Feynman-Kac model associated to the Markov chain \eqref{eq:defNtransESS} and the potential function \eqref{eq:defGESS}. The associated sequence $\{(\Nkac[n],\Nunkac[n])\}_{n \geq 0}$ of Feynman-Kac measures is given for all $\Nfunc[n] \in \Nbanach{n}$ by
\begin{align}
\label{eq:defNkacESS}
&\Nkac_0(\Nfunc[0]) \eqdef \Nunkac_0(\Nfunc[0])= \esp{\Nfunc[0](\epart{0})} \eqsp,  \\
&\Nkac[n](\Nfunc[n])\eqdef \Nunkac[n](\Nfunc[n])/\Nunkac[n](1), \quad \text{for all} \ n \geq 1, \\
&\Nunkac[n](\Nfunc[n]) \eqdef \esp{\Nfunc[n](\epart{n})~\prod_{0\leq p<n}\Npot[p](\epart{p})} \eqsp, \quad \text{for all} \ n \geq 1.
\end{align}
\begin{theorem}
\label{prop:unbiasedESS}
For a particle system $\bx_n =((x_n^1,w_n^1), \dots, (x_n^{N_1},w_n^{N_1})) \in \NXset{n}$ generated by the ESS algorithm and for any $\Nfunc[n] \in \Nbanach{n}$ of the form
$$
\Nfunc[n](\bx_n) = \left(\sum_{i=1}^{N_1} w_n^i\right)^{-1}\sum_{i=1}^{N_1} w_n^i\func[n]\left(x_n^i\right)
$$
where $\func[n] \in \banach{n}$,
\begin{equation}
\Nunkac[n](\Nfunc[n]) = \unkac_n(\func[n]) \quad \mbox{and} \quad \Nkac[n](\Nfunc[n]) = \kac_n(\func[n])\eqsp.
\end{equation}
\end{theorem}
\begin{proof}
See \autoref{subsec:ProofUnbiasedESS}.
\end{proof}

For each $n \in \NN$, let $(\Nepart[1]{n},\dots,\Nepart[N_2]{n}) \in \NXset{n}^{N_2}$ be a population of $N_2$ islands each of $N_1$ individuals. We associate to each island, a weight $\Omega_n^i$, for $i \in \{1, \dots, N_2 \}$. We can also make the islands interact using an ESS criterion.

The process $((\Nepart[1]{n},\Omega_n^1), \dots, (\Nepart[N_2]{n},\Omega_n^{N_2}))$ is a Markov chain which evolves according to selection and mutation steps, defined as follows
\begin{itemize}
\item \textit{Selection step}: if the ESS criterion $\left(\sum_{i=1}^{N_2}\Omega_n^i\Npot[n](\Nepart[i]{n})\right)^2/\sum_{i=1}^{N_2}\left(\Omega_n^i\Npot[n](\Nepart[i]{n})\right)^2$ is larger than $\beta N_2$ for one $\beta \in (0,1)$, we do not resample the islands and we update the weights thanks to the potential function $\Omega_{n+1}^i = \Omega_n^i \Npot[n](\Nepart[i]{n})$; otherwise, we resample the islands with probability proportional to $\{\Omega_n^i\Npot[n](\Nepart[i]{n})\}_{i=1}^{N_2}$ and the weights are all reset to 1.
\item \textit{Mutation step:} each selected island is updated independently according to the Markov transition $\Ntrans{n+1}$.
\end{itemize}
These islands particles allow to define the $N_2$-particle approximation of the measures $\Nkac[n]$ and $\Nunkac[n]$, for any $\Nfunc[n] \in \Nbanach{n}$, as
\begin{align*}
&\NPkac_n(\Nfunc[n])\eqdef \frac{1}{\sum_{i=1}^{N_2}\Omega_n^i}\sum_{i=1}^{N_2} \Omega_n^i \Nfunc[n](\Nepart[i]{n}) \eqsp, \\
&\NPunkac_n(\Nfunc[n])\eqdef \NPkac_n(\Nfunc[n])~\prod_{0\leq p<n}\NPkac_p(\Npot[p]) = \NPkac_n(\Nfunc[n])~\NPunkac_{n}(1) \eqsp.
\end{align*}
\autoref{alg:ESSWithinESS} describes the ESS within ESS island filter.
\begin{algorithm}[!h]
    \begin{algorithmic}[1]
        \caption{\quad ESS within ESS island filter}\label{alg:ESSWithinESS}
        \State \underline{Initialization:}
        \State Set $\boldsymbol{\Omega}_0 = \left( \Omega_0^i \right)_{i=1}^{N_2} = (1, \dots, 1)$.
        \For{$i$ from $1$ to $N_2$}
            \State Set $\Neweight[i]{0} = \left(\Nsmallweight{i,j}{0}\right)_{j=1}^{N_1} = (1,\dots,1)$.
            \State Sample $\Nepart[i]{0}=\left(\Nsmallpart{i,j}{0}\right)_{j=1}^{N_1}$ independently distributed according to $\kac_0$.
        \EndFor
        \For{$p$ from $0$ to $n-1$}
            \State \underline{Island selection step and weight updating:}
            \State Set $N_2^{\mathrm{eff}} = \left(\sum_{i=1}^{N_2}\Omega_p^i\Npot[p](\Nepart[i]{p})\right)^2
            /\sum_{i=1}^{N_2}\left(\Omega_p^i\Npot[p](\Nepart[i]{p})\right)^2$.
            \If{$N_2^{\mathrm{eff}} \geq \alpha_\mathrm{Islands} N_2$}
                \State For $1 \leq i \leq N_2$, set $\Omega_{p+1}^i = \Omega_p^i \Npot[p](\Nepart[i]{p})$.
                \State Set $\boldsymbol{I}_p = (I_p^i)_{i=1}^{N_2} = (1, 2, \dots, N_2)$.
            \Else
                \State Set $\boldsymbol{\Omega}_{p+1} = \left( \Omega_{p+1}^i \right)_{i=1}^{N_2} = (1, \dots, 1)$.
                \State Sample $\boldsymbol{I}_p = (I_p^i)_{i=1}^{N_2}$ multinomially
            with probability proportional to $\left(\Omega_p^i\Npot[p](\Nepart[i]{p},\Neweight[i]{p})\right)_{i=1}^{N_2}$.
            \EndIf
            \State \underline{Island mutation step:}
            \For{$i$ from $1$ to $N_2$}
                \State \underline{Particle selection and weight updating within each island:}
                \State Set $N_1^{\mathrm{eff}} = \left(\sum_{j=1}^{N_1}\omega_p^{I_p^i,j}\pot[p](\Nsmallpart{I_p^i,j}{p})\right)^2/\sum_{j=1}^{N_1}\left(\omega_p^{I_p^i,j}\pot[p](\Nsmallpart{I_p^i,j}{p})\right)^2$.
                \If{$N_1^{\mathrm{eff}} \geq \alpha_\mathrm{Particles} N_1$}
                    \State For $1 \leq j \leq N_1$, set $\omega_{p+1}^{i,j} = \omega_p^{I_p^i,j} \pot[p](\Nsmallpart{I_p^i,j}{p})$.
                    \State Set $\boldsymbol{J}_p^i = (J_p^{i,j})_{j=1}^{N_1} = (1, 2, \dots, N_1)$.
                \Else
                    \State Set $\boldsymbol{\omega}_{p+1}^i = \left( \omega_{p+1}^{i,j} \right)_{j=1}^{N_1} = (1, \dots, 1)$.
                    \State Sample $\boldsymbol{J}_p^i = (J_p^{i,j})_{j=1}^{N_1}$ multinomially
            with probability proportional to $\left(\omega_p^{I_p^i,j}\pot[p](\Nsmallpart{I_p^i,j}{p})\right)_{j=1}^{N_1}$.
                \EndIf
                \State \underline{Particle mutation:}
                \State For $1 \leq j \leq N_1$, sample independently $\Nsmallpart{i,j}{p+1}$ according to $\trans{p+1}(\Nsmallpart{I_p^i,L_p^{i,j}}{p},\cdot)$,
                where $L_p^{i,j}=J_p^{I_p^i,j}$.
            \EndFor
        \EndFor
        \State Approximate $\kac[n](\func[n])$ by $\displaystyle \dfrac{1}{\sum_{i=1}^{N_2} \Omega_n^i} \sum_{i=1}^{N_2} \frac{\Omega_n^i}{\sum_{j=1}^{N_1} \omega_n^{i,j}} \sum_{j=1}^{N_1} \omega_n^{i,j} \func[n]\left(\Nsmallpart{i,j}{n}\right)$.
    \end{algorithmic}
\end{algorithm}

\section{Numerical simulations}
\label{sec:NumericalSimulations}
\begin{example}[Linear Gaussian Model] \label{ex:LGMdef}
In order to assess numerically the previous results, we now consider the Linear Gaussian Model (LGM) defined by:
\begin{equation*}
  X_{p+1} = \phi X_p + \sigma_uU_p\eqsp, \quad Y_p = X_p +  \sigma_vV_p\eqsp,
\end{equation*}
where $X_0\sim\mathcal{N}\left(0,\sigma_u^2/(1-\phi^2)\right)$,
$\left\{U_p\right\}_{p\geq 1}$ and $\left\{V_p\right\}_{p\geq 1}$ are independent sequences of i.i.d.
standard Gaussian random variables, independent of $X_0$.
In the simulations, we have used $n=20$ observations, generated using the model with $\phi = 0.9$, $\sigma_u= 0.6$ and $\sigma_v = 1$.
We focus on the prediction problem, consisting in computing the predictive distribution of the state $X_n$ given  $Y_0, \cdots, Y_{n-1}$. This problem can be cast in the Feynman-Kac framework by setting for all $p \geq 0$
\begin{align*}
&\trans{p+1}(x_p,\rmd x_{p+1}) = \dfrac{1}{\sqrt{2\pi}\sigma_u}\exp\left[-\dfrac{(x_{p+1}-\phi x_p)^2}{2\sigma_u^2}\right]\rmd x_{p+1} \eqsp, \\
&\pot[p](x_p) = \dfrac{1}{\sqrt{2\pi}\sigma_v} \exp\left[-\dfrac{(y_p- x_p)^2}{2\sigma_v^2}\right]\eqsp.
\end{align*}
We estimate the predictive mean of the latent state $\cesp{X_n}{Y_{0},\dots,Y_{n-1}}$. We compare the results obtained for different interactions across the islands and for different values of $N_1$ and $N_2$; in all the simulations, the bootstrap filter is used within the islands.
We have run the simulations independently $250$ times
and we have compared these estimators with the value
computed using the Kalman filter.
Figure \ref{fig:LGMBootstrapWithinAllInteractions} displays the boxplots of the $250$ values
of these estimators.
\begin{figure}[!h]
  \centering
  \vskip-7cm
  \includegraphics[height=22cm,width=\textwidth]{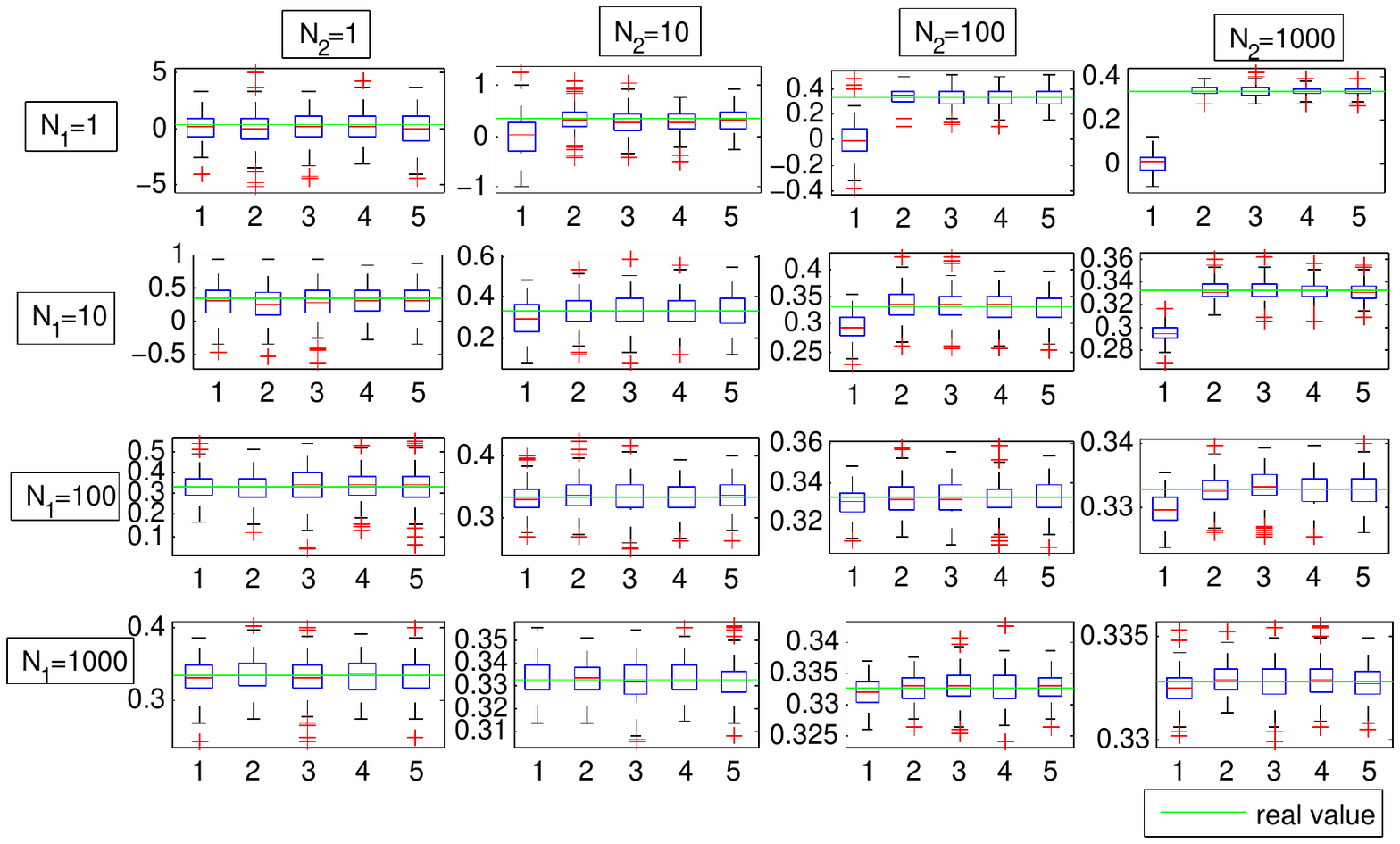}
  \vskip-7cm
  \caption{Comparison of different interactions across the islands with bootstrap within each island for the LGM
  (1) Bootstrap/independent; (2) Bootstrap/ESS; (3) Bootstrap/Bootstrap; (4) Bootstrap/($1/\left\|\pot[n]\right\|_{\infty}$))-bootstrap; (5) Bootstrap/$\mathrm{essup}_{\kac_p^{N_1}}(\pot[n])$-bootstrap}
  \label{fig:LGMBootstrapWithinAllInteractions}
\end{figure}
 As expected, for small values of $N_1$ compared to $N_2$, the bias of independent islands is large compared to cases where islands interact; on the contrary, the variance is smaller for independent islands than for bootstrap island interaction.
 In this example, the type of interaction between islands does not have a significant impact on the dispersion of the estimator (the bias is negligible).

 An important aspect for the efficiency of the algorithms is the number of interactions between islands. The smaller this number is, the quicker the algorithm will be. The number of interactions in the bootstrap case is $nN_2$. We have compared the island interaction number for the $\Peps{p}$-bootstrap and the ESS interactions w.r.t. the bootstrap one, when we apply the bootstrap filter within the islands. We have computed the empirical number of interactions over the $250$ simulations; the results are respectively given in tables \ref{tab:LGMinteractionNumberEpsilonVSBootstrap} and \ref{tab:LGMinteractionNumberESSVSBootstrap}.
\begin{table}
\begin{center}
\begin{tabular}{|c|cc|cc|cc|cc|cc|}
\hline\hline
 \backslashbox{$N_1$}{$N_2$}   & \multicolumn{2}{c|}{ \phantom{.}1\phantom{.}} & \multicolumn{2}{c|}{ \phantom{.}10\phantom{.}}   &    \multicolumn{2}{c|}{\phantom{.}100\phantom{.}} & \multicolumn{2}{c|}{\phantom{.}1000\phantom{.}} \\
\hline
1        &   0   &  20    & 77   & 200   &  825    & 2000  &  8264  & 20000  \\
\hline
10       &   0   &  20    & 47   & 200   &  636    & 2000  &  7122  & 20000  \\
\hline
100      &   0   &  20    & 19   & 200   &  297    & 2000  &  3609  & 20000  \\
\hline
1000     &   0   &  20    &  7   & 200   &  107    & 2000  &  1373  & 20000  \\
\hline\hline
\end{tabular}
\caption{Island interaction number using bootstrap within $\Peps{p}$-bootstrap and double bootstrap for the LGM.}
\label{tab:LGMinteractionNumberEpsilonVSBootstrap}
\end{center}
\end{table}
\begin{table}
\begin{center}
\begin{tabular}{|c|cc|cc|cc|cc|cc|}
\hline\hline
 \backslashbox{$N_1$}{$N_2$}   & \multicolumn{2}{c|}{ \phantom{.}1\phantom{.}} & \multicolumn{2}{c|}{ \phantom{.}10\phantom{.}}   &    \multicolumn{2}{c|}{\phantom{.}100\phantom{.}} & \multicolumn{2}{c|}{\phantom{.}1000\phantom{.}} \\
\hline
1        &   0   &  20    & 86   & 200   &  945    & 2000  &  9056  & 20000  \\
\hline
10       &   0   &  20    & 19   & 200   &  230    & 2000  &  2408  & 20000  \\
\hline
100      &   0   &  20    &  0   & 200   &    0    & 2000  &     0  & 20000  \\
\hline
1000     &   0   &  20    &  0   & 200   &    0    & 2000  &     0  & 20000  \\
\hline\hline
\end{tabular}
\caption{Island interaction number using bootstrap within ESS and double bootstrap for the LGM.}
\label{tab:LGMinteractionNumberESSVSBootstrap}
\end{center}
\end{table}
For a given number of islands, the island interaction number for the ESS and the $\Peps{p}$-bootstrap decrease when the island size grows, whereas it is constant for the bootstrap. The island interaction number is always much smaller using the ESS or the $\Peps{p}$-bootstrap than the bootstrap, across the islands. Moreover, as soon as the number of particles in each island is large enough, the ESS is no longer resampling the islands.

\autoref{theo:localAsymptoEpsilon} assures that the variance is smaller using the $\Peps{p}$-bootstrap than the bootstrap interaction. The variance gain using $\Peps{p}$-bootstrap or ESS instead of bootstrap across the islands is given in table \ref{tab:LGMgainvarianceESSEpsVSBootstrap}. The bootstrap interaction is applied within the islands. The variance is significantly reduced using the $\Peps{p}$-bootstrap or the ESS interaction across the islands, instead of the bootstrap, up to $34$ percent variance reduction.
\begin{table}
\begin{center}
\begin{tabular}{|c|cc|cc|cc|cc|}
\hline\hline
 \backslashbox{$N_1$}{$N_2$}   & \multicolumn{2}{c|}{ \phantom{.}10\phantom{.}}   &    \multicolumn{2}{c|}{\phantom{.}100\phantom{.}} & \multicolumn{2}{c|}{\phantom{.}1000\phantom{.}} \\
\hline
10       &   9.5    & 18.7    & 13.2  & 20.5   & 22.8    & 1.7      \\
\hline
100      &   25.4   & 26.1    & 26.1  & 18.5   & 13.5   & 22.4    \\
\hline
1000     &   28.2   & 34.3    & 19.5  & 33.8   & 25.9   & 26.5     \\
\hline\hline
\end{tabular}
\caption{Percentage of the variance gain using bootstrap within $\Peps{p}$-bootstrap on the left side and ESS within bootstrap on the right side, compared to the double bootstrap, in the LGM example.}
\label{tab:LGMgainvarianceESSEpsVSBootstrap}
\end{center}
\end{table}
\end{example}
\begin{example}[Stochastic volatility model]
\label{ex:reducEpsilonEN}
We consider the stochastic volatility model:
\begin{equation*}
X_{p+1} = \alpha X_p + \sigma U_{p+1}\eqsp, \quad
Y_p = \beta \mathrm{e}^{\frac{X_p}{2}} V_p\eqsp,
\end{equation*}
where $X_0\sim\mathcal{N}\left(0,\sigma^2/(1-\alpha^2)\right)$, $\{U_p\}_{p \geq 0}$ and $\{V_p\}_{p \geq 0}$ are independent sequences of standard Gaussian random variables independent of $X_0$. In the simulations, we have used $n=100$ observations, generated using the model with $\alpha = 0.98$, $\sigma= 0.5$ and $\beta = 1$.
We estimate the predictive mean of the latent state $X_n$ given the observations $Y_0, \cdots, Y_{n-1}$. This problem can be cast in the Feynman-Kac framework by setting for all $p \geq 0$
\begin{align*}
&\trans{p+1}(x_p,\rmd x_{p+1}) = \dfrac{1}{\sqrt{2\pi}\sigma} \exp\left[-\dfrac{(x_{p+1}-\alpha x_p)^2}{2\sigma^2}\right] \rmd x_{p+1} \eqsp, \\
&\pot[p](x_p) = \dfrac{1}{\sqrt{2\pi}\beta}\exp\left[-\dfrac{x_p/2 -y_p^2\mathrm{e}^{-x_p}}{2\beta^2}\right]\eqsp.
\end{align*}
We have computed this quantity using a single run of bootstrap filter with $10^6$ particles.
In the following results, we always consider bootstrap interaction within each island, and we compare different interactions across the islands, for several values of $N_1$ and $N_2$.
We have run the simulations independently $250$ times.
Figure \ref{fig:StoVolplainWithinAllInteractions} displays the boxplots of the $250$ values
of these estimators. The behavior of the different methods is similar to the one observed for the Linear Gaussian Model example.
\begin{figure}[!h]
  \centering
  \vskip-7cm
  \includegraphics[height=22cm,width=\textwidth]{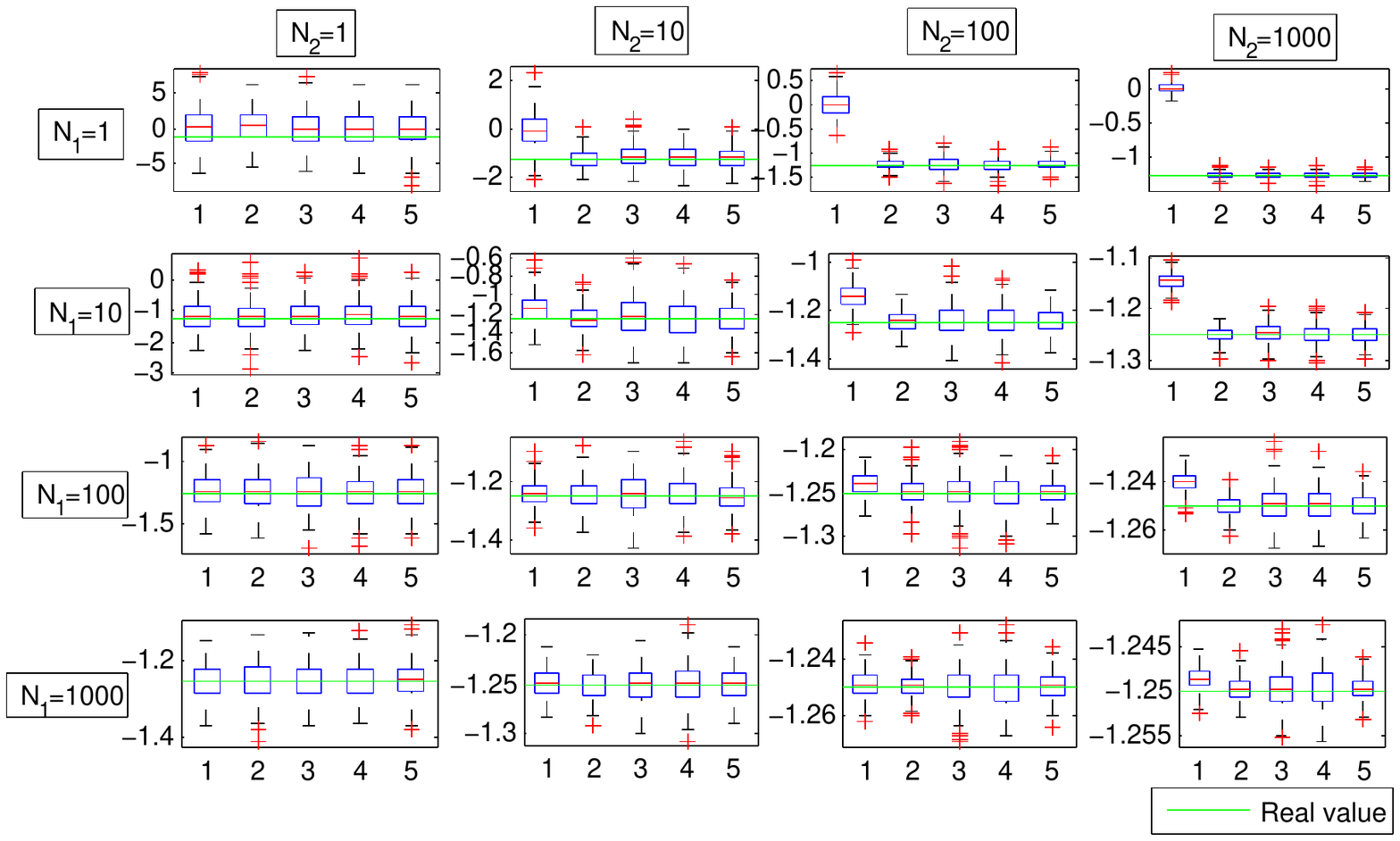}
  \vskip-7cm
  \caption{Comparison of different interactions across the islands with bootstrap within each island for the Stochastic volatility model
  (1) Bootstrap/independent; (2) Bootstrap/ESS; (3) Bootstrap/Bootstrap; (4) Bootstrap/($1/\left\|\pot[n]\right\|_{\infty}$))-bootstrap; (5) Bootstrap/$\mathrm{essup}_{\kac_p^{N_1}}(\pot[n])$-bootstrap}
  \label{fig:StoVolplainWithinAllInteractions}
\end{figure}

We have compared the island interaction number for the $\Peps{p}$-bootstrap and the ESS interactions w.r.t. the bootstrap one, when we apply the bootstrap filter within the islands. We have computed the empirical number of interactions over the $250$ simulations; the results are respectively given in tables \ref{tab:StoVolinteractionNumberEpsilonVSBootstrap} and  \ref{tab:StoVolinteractionNumberESSVSBootstrap}. The number of interactions in the bootstrap case is $nN_2$. The same phenomena are observed as for the Linear Gaussian Model example.
\begin{table}
\begin{center}
\begin{tabular}{|c|cc|cc|cc|cc|cc|}
\hline\hline
 \backslashbox{$N_1$}{$N_2$}   & \multicolumn{2}{c|}{ \phantom{.}1\phantom{.}} & \multicolumn{2}{c|}{ \phantom{.}10\phantom{.}}   &    \multicolumn{2}{c|}{\phantom{.}100\phantom{.}} & \multicolumn{2}{c|}{\phantom{.}1000\phantom{.}} \\
\hline
1        &    0   &  100    & 332   & 1000   &  4021    & 10000  &  42185   & 100000  \\
\hline
10       &    0   &  100    & 221   & 1000   &  3069    & 10000  &  34789   & 100000  \\
\hline
100      &    0   &  100    & 100   & 1000   &  1523    & 10000  &  18647   & 100000  \\
\hline
1000     &    0   &  100    &  36   & 1000   &   577    & 10000  &   7332   & 100000  \\
\hline\hline
\end{tabular}
\caption{Island interaction number using bootstrap within $\Peps{p}$-bootstrap and double bootstrap for the Stochastic volatility model.}
\label{tab:StoVolinteractionNumberEpsilonVSBootstrap}
\end{center}
\end{table}

\begin{table}
\begin{center}
\begin{tabular}{|c|cc|cc|cc|cc|cc|}
\hline\hline
 \backslashbox{$N_1$}{$N_2$}   & \multicolumn{2}{c|}{ \phantom{.}1\phantom{.}} & \multicolumn{2}{c|}{ \phantom{.}10\phantom{.}}   &    \multicolumn{2}{c|}{\phantom{.}100\phantom{.}} & \multicolumn{2}{c|}{\phantom{.}1000\phantom{.}} \\
\hline
1        &   0   &  100    &  301   & 1000   &  3514   & 10000  &  36108  & 100000  \\
\hline
10       &   0   &  100    &  109   & 1000   &  1229   & 10000  &  12096  & 100000  \\
\hline
100      &   0   &  100    &   15   & 1000   &   186   & 10000  &   1956  & 100000  \\
\hline
1000     &   0   &  100    &    0   & 1000   &    0    & 10000  &      0  & 100000  \\
\hline\hline
\end{tabular}
\caption{Island interaction number using bootstrap within ESS and double bootstrap for the Stochastic volatility example.}
\label{tab:StoVolinteractionNumberESSVSBootstrap}
\end{center}
\end{table}

The variance gain using the $\Peps{p}$-bootstrap or the ESS instead of the bootstrap across the islands is given in table \ref{tab:StoVolgainvarianceplainEpsVSBootstrap}. The bootstrap interaction is applied within the islands. The variance is significantly reduced using the $\Peps{p}$-bootstrap or the ESS interaction across the islands, instead of the bootstrap, up to $66$ percent variance reduction.
\begin{table}
\begin{center}
\begin{tabular}{|c|cc|cc|cc|cc|}
\hline\hline
 \backslashbox{$N_1$}{$N_2$}   & \multicolumn{2}{c|}{ \phantom{.}10\phantom{.}}   &    \multicolumn{2}{c|}{\phantom{.}100\phantom{.}} & \multicolumn{2}{c|}{\phantom{.}1000\phantom{.}} \\
\hline
10       &   44.2    & 57.8     & 35.3   & 57.2     & 30.4    & 50.7   \\
\hline
100      &   46.4    & 49.3     & 52.2   & 44.6     & 46.8    & 65    \\
\hline
1000     &   30.4    & 41.7     & 49.6   & 66.9     & 55.8    & 61.4  \\
\hline\hline
\end{tabular}
\caption{Percentage of the variance gain using bootstrap within $\Peps{p}$-bootstrap on the left side and ESS within bootstrap on the right side, compared to the double bootstrap, in the Stochastic volatility example.}
\label{tab:StoVolgainvarianceplainEpsVSBootstrap}
\end{center}
\end{table}
\end{example}

\section{Proofs}
\label{sec:proofs}
	\subsection{Proof of \autoref{theo:unbiased}}
	\label{subsec:ProofUnbiased}
Using \eqref{eq:approxDef_1} and \eqref{eq:defG}, $\Npot[n](\epart{n})$ may be expressed as $\Npot[n](\epart{n}) = \Pkac_n(\pot[n]) $ where $\epart{n}$ and $\Pkac_n$ are defined in \eqref{eq:approxDef_1} and \eqref{eq:defSingleParticles}, respectively.
Similarly, for any $\Nfunc[n] \in \Nbanach{n}$ of the form $\Nfunc[n](\bx_n) = N_1^{-1} \sum_{i=1}^{N_1} \func[n](x_n^i) $ where $\func[n] \in \banach{n}$, $\Nfunc[n](\epart{n})$ is given by
$\Nfunc[n](\epart{n})= \Pkac_n(f_n)$. Note that
\begin{align}
\label{eq:assumptionBoldM}
\Nunkac_n(\Nfunc[n]) \eqdef \esp{\Nfunc[n](\epart{n})~\prod_{0\leq p<n}\Npot[p](\epart{p})}
&= \esp{\Pkac_n(\func[n])~\prod_{0\leq p<n}\Pkac_p(\pot[p])},
\end{align}
and since by \eqref{eq:approxDef_2}, it suffices to prove that $\unkac_n^{N_1}(\func[n])$ is an unbiased estimator of $\unkac_{n}(\func[n])$, i.e.
\begin{equation}
\label{eq:unbiased}
\esp{\unkac_n^{N_1}(\func[n])}=\unkac_{n}(\func[n]).
\end{equation}
Define the filtration $\mathcal{F}^{N_1}_n \eqdef \sigma \left( \epart{p}, 0 \leq p \leq n \right) \eqsp.$
Note that
\begin{multline}
\label{eq:kac-expectation}
\cesp{\Pkac_p(\func[p])}{\mathcal{F}^{N_1}_{p-1}}=\frac{1}{N_1}\sum_{i=1}^{N_1}{\cesp{\func[p](\epart[i]{p})}{\mathcal{F}^{N_1}_{p-1}}}=\cesp{\func[p](\epart[1]{p})}{\mathcal{F}^{N_1}_{p-1}} \\
= \dfrac{\sum_{i=1}^{N_1}{\pot[p-1](\epart[i]{p-1}) \trans{p} \func[p] (\epart[i]{p-1})}}{\sum_{i=1}^{N_1}{\pot[p-1](\epart[i]{p-1})}} = \dfrac{\Pkac_{p-1}(Q_{p}\func[p])}{\Pkac_{p-1}(\pot[p-1])} \eqsp,
\end{multline}
where $Q_{p}$ is defined in \eqref{eq:defQ}.

By the definition of $\Punkac_n$ given in \eqref{eq:approxDef_2}, we have
\begin{align*}
\esp{\Punkac_n(\func[n])} &= \esp{\cesp{\Pkac_n(\func[n])}{\mathcal{F}^{N_1}_{n-1}}~\prod_{0\leq p<n}\Pkac_p(\pot[p])} \\
&= \esp{\dfrac{\Pkac_{n-1}(\semi_n \func[n])}{\Pkac_{n-1}(\pot[n-1])} ~\prod_{0\leq p<n}\Pkac_p(\pot[p])} \\
&= \esp{\Pkac_{n-1}(\semi_n \func[n])~\prod_{0\leq p<n-1}\Pkac_p(\pot[p])} \eqsp.
\end{align*}
By iterating this step we get
\begin{align*}
\esp{\Punkac_n(\func[n])} &= \esp{\Pkac_{0}(\semi_1 \cdots \semi_n \func[n])}
= \esp{\semi_1 \cdots \semi_n\func[n](\epart[1]{0})} \\
&= \unkac_0\semi_1 \cdots \semi_n\func[n]
= \unkac_n(\func[n]) \eqsp.
\end{align*}

	\subsection{Proof of \autoref{theo:singleIslandAsymptotics}}
	\label{subsec:ProofSingleIsland}
We preface the proof by the following Lemma.
\begin{lemma}\label{lem:LimWetaWgamma}
For any $\func[n]^1,\func[n]^2\in\banach{n}$, the pair $(W_n^{\unkac,N_1}(\func[n]^1),W_n^{\kac,N_1}(\func[n]^2))$ converges in law, as $N_1$ tends to infinity, to $(W_n^{\unkac}(\func[n]^1),W_n^{\kac}(\func[n]^2))$. In addition, for any polynomial function $\Phi:\RR^2 \rightarrow \RR$, we have:
$$
\lim_{N_1 \rightarrow \infty} \esp{\Phi\left(W_n^{\unkac,N_1}(\func[n]^1),W_n^{\kac,N_1}(\func[n]^2)\right)} = \esp{\Phi\left(W_n^{\unkac}(\func[n]^1),W_n^{\kac}(\func[n]^2)\right)} \eqsp.
$$
\end{lemma}

\begin{proof}
For any $(\alpha,\beta)\in\RR^2$ by the definitions \eqref{eq:defWgamma} of $W_n^{\unkac,N_1}$ and \eqref{eq:defWeta} of $W_n^{\kac,N_1}$ we have
\begin{multline*}
\alpha W_n^{\unkac,N_1}(\func[n]^1) + \beta W_n^{\kac,N_1}(\func[n]^2) \\
= \sum_{p=0}^n \left[ \alpha \Punkac_p(1) W_p^{N_1}(\semi_{p,n}\func[n]^1) + \beta \frac{\Punkac_p(1)}{\Punkac_n(1)}W_p^{N_1}(\semi_{p,n}(\func[n]^2-\kac[n](\func[n]^2)))\right] \eqsp.
\end{multline*}
As in the proof of Theorem \ref{theo:localAsympto}, a simple application of Slutsky's Lemma allows to show that $\alpha W_n^{\unkac,N_1}(\func[n]^1) + \beta W_n^{\kac,N_1}(\func[n]^2)$ converges in law to $\alpha W_n^{\unkac}(\func[n]^1) + \beta W_n^{\kac}(\func[n]^2)$. The proof follows from \cite[Theorem 7.4.4]{delmoral:2004}, using that for any $p \geq 1$,
\begin{align}
\label{eq:LpMeanErrorGamma}
&\sup_{N_1\geq 1} \esp{\left|W_n^{\unkac,N_1}(\func[n]^1)\right|^p}^{1/p} \leq c_p(n) ||\func[n]^1|| \eqsp, \\
\label{eq:LpMeanErrorEta}
&\sup_{N_1\geq 1} \esp{\left|W_n^{\kac,N_1}(\func[n]^2)\right|^p}^{1/p} \leq c_p(n) ||\func[n]^2|| \eqsp,
\end{align}
for some finite constant $c_p(n)$ depending only on $p$ and $n$.
\end{proof}

\begin{proof}of \autoref{theo:singleIslandAsymptotics}
Consider first the bias term. We decompose the error as follows using \eqref{eq:defWeta}:
\begin{align*}
&N_1 \left[\Pkac_n(\func[n]) - \kac_n(\func[n]) \right] = \sqrt{N_1} W_n^{\eta,N_1}(\func[n])
= \sqrt{N_1} \dfrac{\unkac_n(1)}{\Punkac_n(1)}W_n^{\gamma,N_1}\left( \dfrac{\func[n] - \kac_n(\func[n])}{\unkac_n(1)}\right) \\
& = \sqrt{N_1} \left[\dfrac{\unkac_n(1)}{\Punkac_n(1)}-1\right]  W_n^{\gamma,N_1}\left( \dfrac{\func[n] - \kac_n(\func[n])}{\unkac_n(1)}\right) + \sqrt{N_1} W_n^{\gamma,N_1}\left( \dfrac{\func[n] -\kac_n(\func[n])}{\unkac_n(1)}\right)\eqsp.
\end{align*}
Since $W_n^{\gamma,N_1} =\sqrt{N_1} \left[\Punkac_n-\unkac_n \right] $, \autoref{theo:unbiased} shows that, the expectation of the second term on the RHS of the previous equation, is zero.
By noting that
\begin{equation*}
\left[\dfrac{\unkac_n(1)}{\Punkac_n(1)}-1\right] =
 -\dfrac{1}{\Punkac_n(1)} [\Punkac_n-\unkac_n](1) = -\frac{1}{\sqrt{N_1}}\dfrac{1}{\Punkac_n(1)} W_n^{\gamma,N_1}(1) \eqsp,
\end{equation*}
where $W_n^{\unkac,N_1}$ is defined in \eqref{eq:defWgamma}, we get
\begin{align*}
N_1 \esp{\Pkac_n(\func[n])-\kac_n(\func[n])}
&= -\esp{\dfrac{1}{\Punkac_n(1)} W_n^{\unkac,N_1}(1) W_n^{\unkac,N_1}\left(\dfrac{\func[n]-\kac_n(\func[n])}{\unkac_n(1)}\right)} \\
&= -\dfrac{1}{\unkac_n(1)}\esp{ W_n^{\unkac,N_1}(1) W_n^{\kac,N_1}(\func[n])}\eqsp,
\end{align*}
where $W_n^{\kac,N_1}$ is given in \eqref{eq:defWeta}. According to \autoref{lem:LimWetaWgamma}:
\begin{equation}\label{eq:biaisproof}
\lim_{N_1 \rightarrow \infty} N_1 \esp{\Pkac_n(\func[n]) - \kac_n(\func[n])} = -\dfrac{1}{\unkac_n(1)}\esp{ W_n^{\unkac}(1) W_n^{\kac}(\func[n])} = B_n(\func[n])\eqsp,
\end{equation}
by the definitions of $W_n^{\unkac}$ and $W_n^{\kac}$. Consider now the variance. We use the decomposition
\begin{equation*}
\var{\Pkac_n(\func[n])}  = \esp{ \left( \Pkac_n(\func[n]) - \kac_n(\func[n]) \right)^2} - \left\{\esp{ \Pkac_n(\func[n]) - \kac_n(\func[n]) }\right\}^2\eqsp.
\end{equation*}
Using \eqref{eq:biaisproof}, we get $\left\{\esp{ \Pkac_n(\func[n]) - \kac_n(\func[n]) }\right\}^2= O(N_1^{-2})$.
From the definition~\eqref{eq:defWeta} of $W_n^{\kac,N_1}$, it follows $\esp{ \left( \Pkac_n(\func[n]) - \kac_n(\func[n]) \right)^2} = N_1^{-1} \esp{W_n^{\kac,N_1}(\func[n])^2}$, implying that
$\lim_{N_1 \rightarrow \infty} N_1 \var{\Pkac_n(\func[n])} = \esp{W_n^{\kac}(\func[n])^2} = V_n(\func[n])$, by the definition of $W_n^{\kac}$ and using again \autoref{lem:LimWetaWgamma}.
\end{proof}

	\subsection{Proof of \autoref{theo:InteractingAsymptotic}}
	\label{subsec:ProofInteractingIslands}
	We preface the proof of \autoref{theo:InteractingAsymptotic} by the following result on the usual Feynman-Kac model.

\begin{lemma}
\label{lem:espTwoFunc}
For any time horizon $n\geq 0$ and any functions $\func[n]^1, \func[n]^2 \in \banach{n}$ such that $\kac_n(\func[n]^1) = 0$, we have
\begin{multline*}
\lim_{N_1 \rightarrow \infty} N_1\esp{\Pkac_n(\func[n]^1) \Pkac_n(\func[n]^2) \prod_{p=0}^{n-1} \Pkac_p(\pot[p])} \\
= \unkac_n(1) \sum_{p=0}^n \esp{W_p(\semiNorm_{p,n}(\func[n]^1))W_p(\semiNorm_{p,n}(\func[n]^2-\kac_n(\func[n]^2)))}\eqsp.
\end{multline*}
\end{lemma}
\begin{proof}
By the definition \eqref{eq:approxDef_2} of $\Punkac_n$ we have $\Pkac_n(\func[n]^1) \prod_{p=0}^{n-1} \Pkac_p(\pot[p]) = \Punkac_n(\func[n]^1)$, and, according to \eqref{eq:unbiased}, $\esp{\Punkac_n(\func[n]^1)} = \unkac_n(\func[n]^1) = \unkac_n(1) \kac_n(\func[n]^1)=0$, so that we get
\begin{align*}
&\esp{\Pkac_n(\func[n]^1) \Pkac_n(\func[n]^2) \prod_{p=0}^{n-1} \Pkac_p(\pot[p])} = \esp{\Punkac_n(\func[n]^1) \Pkac_n(\func[n]^2)} \\
 &= \esp{\Punkac_n(\func[n]^1) \left(\Pkac_n(\func[n]^2)-\kac_n(\func[n]^2)\right)} + \esp{\Punkac_n(\func[n]^1)}\kac_n(\func[n]^2)  \\
 &= \esp{\left(\Punkac_n(\func[n]^1)-\unkac_n(\func[n]^1)\right) \left(\Pkac_n(\func[n]^2)-\kac_n(\func[n]^2)\right)} = \dfrac{1}{N_1} \esp{W_n^{\unkac,N_1}(\func[n]^1)W_n^{\kac,N_1}(\func[n]^2)}\eqsp.
\end{align*}
Then, Lemma \ref{lem:LimWetaWgamma} gives
$$\lim_{N_1 \rightarrow \infty} N_1\esp{\Pkac_n(\func[n]^1) \Pkac_n(\func[n]^2) \prod_{p=0}^{n-1} \Pkac_p(\pot[p])} = \esp{W_n^{\kac}(\func[n]^2) W_n^{\unkac}(\func[n]^1)}\eqsp,$$
where $W_n^{\unkac}$ and $W_n^{\kac}$ are given by \eqref{eq:defWGammaLim} and \eqref{eq:defWetaLim}.
\end{proof}

\begin{lemma}
\label{lem:semiLinear}
For any time horizon $n \geq 1$, and any linear function $\Nfunc[n]\in \Nbanach{n}$ of the form
 $$
\Nfunc[n](\bx_{n})=\omeas^{N_1}\func[n](\bx_{n})\eqsp, \quad \text{where}~\func[n] \in \banach{n}\eqsp,
$$
we have
\begin{align}
\label{eq:RelSemi}
&\Nsemi[n]\Nfunc[n](\bx_{n-1})=\omeas^{N_1}\semi_n \func[n](\bx_{n-1})\eqsp, \\
\label{eq:RelSemibis}
&\Nsemi_{p,n}\Nfunc[n](\bx_{p})=\omeas^{N_1}\semi_{p,n} \func[n](\bx_{p}) \eqsp, && \text{for any $p\leq n$} \eqsp, \\
&\NsemiNorm_{p,n}\Nfunc[n](\bx_{p})=\omeas^{N_1}\semiNorm_{p,n} \func[n](\bx_{p}) &&\text{for any $p\leq n$}  \eqsp.
\end{align}
\end{lemma}
\begin{proof}
We have from \eqref{eq:kac-expectation}
\begin{multline*}
\Ntrans{n}\Nfunc[n](\bx_{n-1})=\cesp{\Nfunc_n(\epart{n})}{\epart{n-1}=\bx_{n-1}}\\=\cesp{\omeas^{N_1}\func[n](\epart{n})}{\epart{n-1}=\bx_{n-1}} = \dfrac{\omeas^{N_1}\semi[n]\func[n](\bx_{n-1})}{\omeas^{N_1}\pot[n-1](\bx_{n-1})}\eqsp,
\end{multline*}
which implies
\begin{multline*}
\Nsemi_n \Nfunc[n](\bx_{n-1})=\Npot[n-1](\bx_{n-1}) ~\Ntrans{n}\Nfunc[n](\bx_{n-1})\\
=\omeas^{N_1}\pot[n-1](\bx_{n-1})\times\frac{\omeas^{N_1}\semi_n\func[n](\bx_{n-1})}{\omeas^{N_1}\pot[n-1](\bx_{n-1})} \eqsp,
\end{multline*}
showing \eqref{eq:RelSemi}.
The proof of \eqref{eq:RelSemibis} follows by an induction since
$$
\Nsemi[p,n]\Nfunc[n](\bx_{p})=\Nsemi[p,n-1]\Nsemi[n]\Nfunc[n](\bx_{p}) \eqsp.
$$
\end{proof}

\begin{proof}of \autoref{theo:InteractingAsymptotic}
~~\\
\textbf{Asymptotic bias behavior:}
For any fixed $N_1$, the asymptotic bias behavior of $\PNkac[n](\Nfunc[n])$ is given for any $\Nfunc[n] \in \Nbanach{n}$ by applying \autoref{theo:singleIslandAsymptotics} to the island particle model in the bootstrap case:
$$\lim_{N_2 \rightarrow \infty} N_2\esp{\PNkac[n](\Nfunc[n])-\Nkac_n(\Nfunc[n])} = -\sum_{p=0}^{n} \Nkac_p\left[ \NsemiNorm_{p,n}(1)\NsemiNorm_{p,n} \left( \Nfunc[n]-\Nkac_n(\Nfunc[n])\right)\right] \eqsp.$$
For linear functions $\Nfunc[n]$ of the form $\Nfunc[n] = \omeas^{N_1}\func[n]$ where $\func[n] \in \banach{n}$, Lemma \ref{lem:semiLinear} states that
\begin{multline}\label{eq:NTransFtoTransf}
\NsemiNorm_{p,n} \left(\Nfunc[n]-\Nkac_n(\Nfunc[n])\right)(\epart{p}) = \omeas^{N_1}\semiNorm_{p,n} \left( \func[n]-\kac_n(\func[n])\right)(\epart{p})\\
 = \Pkac_p(\semiNorm_{p,n} \left(\func[n]-\kac_n(\func[n])\right)) \eqsp,
 \end{multline}
and
\begin{equation}\label{eq:Ntrans1toTrans1}
\NsemiNorm_{p,n} (1)(\epart{p}) = \omeas^{N_1}\semiNorm_{p,n} \left( 1\right)(\epart{p})
\\
= \Pkac_p(\semiNorm_{p,n} \left( 1\right)) \eqsp.
\end{equation}
Therefore, we get
\begin{align} \label{eq:devBoldEta}
&\Nkac_p\left[ \NsemiNorm_{p,n}(1)\NsemiNorm_{p,n} \left( \Nfunc[n]-\Nkac_n(\Nfunc[n])\right)\right] \eqnum{1}  \dfrac{\Nunkac_p\left[ \NsemiNorm_{p,n}(1)\NsemiNorm_{p,n} \left( \Nfunc[n]-\Nkac_n(\Nfunc[n])\right)\right]}{\Nunkac_p(1)} \\ \nonumber
& \phantom{\Nkac_p[ \NsemiNorm_{p,n}(1)} \eqnum{2} \dfrac{\esp{\NsemiNorm_{p,n}(1)(\epart{p})\NsemiNorm_{p,n} \left(\Nfunc[n]-\Nkac_n(\Nfunc[n])\right)(\epart{p}) \prod_{\ell=0}^{p-1} \Npot_\ell(\epart{\ell})}}{\Nunkac_p(1)} \\
\nonumber
& \phantom{\Nkac_p[ \NsemiNorm_{p,n}(1)}
\eqnum{3} \dfrac{\esp{\Pkac_p(\semiNorm_{p,n} \left( \func[n]-\kac_n(\func[n])\right)) \Pkac_p(\semiNorm_{p,n} \left( 1\right)) \prod_{\ell=0}^{p-1} \kac_\ell^{N_1}(\pot[\ell])}}{\unkac_p(1)} \eqsp,
\end{align}
where $(1)$ is simply the definition \eqref{eq:defNEta} of $\Nkac_p$, $(2)$ stems from the definition \eqref{eq:interpretation-feynman-kac-Nflow} of $\Nunkac_p$, and $(3)$ follows from \autoref{theo:unbiased}, the definition \eqref{eq:defG} of $(\Npot_\ell)_{\ell \geq 0}$ and equations \eqref{eq:NTransFtoTransf} and \eqref{eq:Ntrans1toTrans1}.
As $\kac_p(\semiNorm_{p,n} \left( \func[n]-\kac_n(\func[n])\right)) = 0$ we can apply  \autoref{lem:espTwoFunc} and
\begin{align*}
& \lim_{N_1 \rightarrow \infty} N_1 \Nkac_p\left[ \NsemiNorm_{p,n}(1)\NsemiNorm_{p,n} \left( \Nfunc[n]-\Nkac_n(\Nfunc[n])\right)\right] \\
& \phantom{\lim_{N_1 \rightarrow \infty} N_1 \Nkac_p} = \sum_{\ell=0}^p \esp{W_\ell(\semiNorm_{\ell,p}(\semiNorm_{p,n}(1)-\kac_p \semiNorm_{p,n}(1))) W_\ell(\semiNorm_{\ell,p}\semiNorm_{p,n}(\func[n]-\kac_n(\func[n])))} \\
& \phantom{\lim_{N_1 \rightarrow \infty} N_1 \Nkac_p} = \sum_{\ell=0}^p \esp{W_\ell(\semiNorm_{\ell,n}(1)-\semiNorm_{\ell,p}(1)) W_\ell(\semiNorm_{\ell,n}(\func[n]-\kac_n(\func[n])))} \eqsp,
\end{align*}
from which we conclude that
\begin{align*}
& \lim_{N_1 \rightarrow \infty} \lim_{N_2 \rightarrow \infty} N_1 N_2\esp{\PNkac[n](\Nfunc[n])-\Nkac_n(\Nfunc[n])} \\
& \phantom{\lim_{N_1 \rightarrow \infty} } = -\sum_{p=0}^{n} \sum_{\ell=0}^p \esp{W_\ell(\semiNorm_{\ell,n}(1)-\semiNorm_{\ell,p}(1)) W_\ell(\semiNorm_{\ell,n}(\func[n]-\kac_n(\func[n])))} \\
& \phantom{\lim_{N_1 \rightarrow \infty} } = -\sum_{\ell=0}^{n} \sum_{p=\ell}^n \esp{W_\ell(\semiNorm_{\ell,n}(1)-\semiNorm_{\ell,p}(1)) W_\ell(\semiNorm_{\ell,n}(\func[n]-\kac_n(\func[n])))} \\
& \phantom{\lim_{N_1 \rightarrow \infty} } = B_n(\func[n]) + \widetilde{B}_n(\func[n])\eqsp,
\end{align*}
where $B_n(\func[n])$ is defined in \eqref{eq:defBn} and $\widetilde{B}_n(\func[n])$ is given in \eqref{eq:defBtilden}.
~~\\

\textbf{Asymptotic variance behavior:}
For any fixed $N_1$, the asymptotic variance behavior of $\PNkac[n](\Nfunc[n])$ is given for any $\Nfunc[n] \in \Nbanach{n}$ by applying \autoref{theo:singleIslandAsymptotics} to the island particle model in the bootstrap case:
$$\lim_{N_2 \rightarrow \infty} N_2\var{\PNkac[n](\Nfunc[n])} = \sum_{p=0}^{n} \Nkac_p\left[ \NsemiNorm_{p,n} \left( \Nfunc[n]-\Nkac_n(\Nfunc[n])\right)^2\right] \eqsp.$$
For linear functions $\Nfunc[n]$ of the form $\Nfunc[n] = \omeas^{N_1}\func[n]$ where $\func[n] \in \banach{n}$, using the same steps as in \eqref{eq:devBoldEta}, we get
$$\Nkac_p\left[ \NsemiNorm_{p,n} \left( \Nfunc[n]-\Nkac_n(\Nfunc[n])\right)^2\right] = \dfrac{\esp{\Pkac_p(\semiNorm_{p,n} \left( \func[n]-\kac_n(\func[n])\right))^2 \prod_{\ell=0}^{p-1} \kac_\ell^{N_1}(\pot[\ell])}}{\unkac_p(1)} \eqsp.$$
As $\kac_p(\semiNorm_{p,n} \left(\func[n]-\kac_n(\func[n])\right)) = 0$ we can apply \autoref{lem:espTwoFunc} and
\begin{multline*}
\lim_{N_1 \rightarrow \infty} N_1 \Nkac_p\left[ \NsemiNorm_{p,n} \left( \Nfunc[n]-\Nkac_n(\Nfunc[n])\right)^2\right]
= \sum_{\ell=0}^p \esp{W_\ell(\semiNorm_{\ell,p}\semiNorm_{p,n}(\func[n]-\kac_n(\func[n])))^2} \\
= \sum_{\ell=0}^p \esp{W_\ell(\semiNorm_{\ell,n}(\func[n]-\kac_n(\func[n])))^2} \eqsp,
\end{multline*}
from which we conclude that
\begin{multline*}
\lim_{N_1 \rightarrow \infty} \lim_{N_2 \rightarrow \infty} N_1 N_2\var{\PNkac[n](\Nfunc[n])} = \sum_{p=0}^{n} \sum_{\ell=0}^p \esp{W_\ell(\semiNorm_{\ell,n}(\func[n]-\kac_n(\func[n])))^2} \\
= \sum_{\ell=0}^{n} \sum_{p=\ell}^n \esp{W_\ell(\semiNorm_{\ell,n}(\func[n]-\kac_n(\func[n])))^2}
= V_n(\func[n]) + \widetilde{V}_n(\func[n])\eqsp,
\end{multline*}
where $V_n(\func[n])$ is defined in \eqref{eq:defVn} and $\widetilde{V}_n(\func[n])$ is given in \eqref{eq:defVtilden}.
\end{proof}

	\subsection{Proof of \autoref{ProofUnbiasedEpsilon}}
	\label{subsec:ProofUnbiasedEpsilon}
\begin{lemma}
\label{lem:propertyS}
Let $ \Peps{n}$ be a nonnegative constant such that $ \Peps{n}~\pot[n]\in [0,1]$.
Then
\begin{equation*}
\PPsi{n}(\mu_{n})=\mu_n S_{n,\mu_n} \eqsp,
\end{equation*}
where $S_{n,\mu_n}$ is defined in \eqref{eq:defS}.
\end{lemma}
\begin{proof}
By \eqref{eq:defS} and \eqref{eq:defPsi} we have for any $A_n \in \Xfield{n}$
\begin{align*}
&\mu_n S_{n,\mu_n}(A_n) = \int \mu_n(\rmd x_n) S_{n,\mu_n}(x_n,A_n) \\
&\phantom{\mu_n S_{n,\mu_n}(A_n)} = \int \mu_n(\rmd x_n) \left[ \Peps{n}~\pot[n](x_n) \delta_{x_n}(A_n)+ \left(1- \Peps{n}~\pot[n](x_n)\right)\PPsi{n}(\mu_{n})(A_n) \right] \\
&\phantom{\mu_n S_{n,\mu_n}(A_n)} = \Peps{n} \int_{A_n} \mu_n(\rmd x_n)\pot[n](x_n) + \left(1- \Peps{n}~\mu_{n}(\pot[n])\right)\PPsi{n}(\mu_{n})(A_n) \\
&\phantom{\mu_n S_{n,\mu_n}(A_n)} = \Peps{n} \mu_{n}(\pot[n])\PPsi{n}(\mu_{n})(A_n) + \left(1- \Peps{n}~\mu_{n}(\pot[n])\right)\PPsi{n}(\mu_{n})(A_n)=\PPsi{n}(\mu_{n})(A_n) \eqsp. \qed
\end{align*}
\end{proof}
Let $\mathcal{F}^{N_1}_n$ be the increasing filtration associated to the particle evolution
$\mathcal{F}^{N_1}_n \eqdef \sigma \left( \epart{p}, 0 \leq p \leq n \right) \eqsp.$
As in the proofs of \autoref{theo:unbiased} and \autoref{prop:unbiasedESS}, the only point is to prove that
$$
\cesp{\Pkac_p(\func[p])}{\mathcal{F}^{N_1}_{p-1}}=\dfrac{\Pkac_{p-1}(Q_{p}\func[p])}{\Pkac_{p-1}(\pot[p-1])} \eqsp,
$$
where $Q_{p}$ is defined in \eqref{eq:defQ}. Or,
\begin{multline}
\cesp{\Pkac_p(\func[p])}{\mathcal{F}^{N_1}_{p-1}}=\frac{1}{N_1}\sum_{i=1}^{N_1}{\cesp{\func[p](\epart[i]{p})}{\mathcal{F}^{N_1}_{p-1}}}= \frac{1}{N_1}\sum_{i=1}^{N_1} \Ntrans{p}(\func[p])(\epart[i]{p-1}) \\
=\Pkac_{p-1}\Ntrans{p}(\func[p]) = \Pkac_{p-1} S_{p-1,\Pkac_{p-1}} \trans{p}(\func[p]) = \PPsi{p-1}(\Pkac_{p-1})\trans{p}(\func[p]) = \dfrac{\Pkac_{p-1}(Q_{p}\func[p])}{\Pkac_{p-1}(\pot[p-1])} \eqsp,
\end{multline}
using respectively \eqref{eq:approxDef_1}, \eqref{eq:defNtransEpsilon}, \autoref{lem:propertyS} and \eqref{eq:defPsi}.
		
	\subsection{Proof of \autoref{cor:choice-epsilon}}
	\label{subsec:ProofEpsilon}
	For the $\epsilon$-interaction bootstrap, the sequence $(W_p^{N_1})_{1 \leq p \leq n}$ converges in law, as $N_1$ tends to infinity, to a sequence of $n$ independent centered Gaussian random fields $(W_p)_{0 \leq p \leq n}$ with variance given by
\begin{align*}
\esp{W_p(\func[p])^2}&=\kac_{p-1} S_{p-1,\kac_{p-1}}\trans{p}\func[p]^2-\kac_{p-1} \left[ \left(S_{p-1,\kac_{p-1}} \trans{p}\func[p]\right)^2\right] \\
&=\PPsi{p-1}(\kac_{p-1})(\trans{p}\func[p]^2)-\kac_{p-1} \left[ \left(S_{p-1,\kac_{p-1}} \trans{p}\func[p]\right)^2\right] \eqsp,
\end{align*}
thanks to \autoref{lem:propertyS}.
~~\\
In the special case $\Peps{p}=0$ (the bootstrap case), the function $S_{p,\kac_p}g_p$ is constant and equal to $ \PPsi{p}(\kac_{p})(g_p)$ and the variance for the bootstrap is just
$$
\PPsi{p-1}(\kac_{p-1})(\trans{p}\func[p]^2)- \left(\PPsi{p-1}(\kac_{p-1}) \trans{p}\func[p]\right)^2
$$
Therefore, the variance of the $\epsilon$-interaction bootstrap may be decomposed as follows
\begin{align*}
\esp{W_p(\func[p])^2} &= \left(\PPsi{p-1}(\kac_{p-1})(\trans{p}\func[p]^2)- \left(\PPsi{p-1}(\kac_{p-1})\trans{p}\func[p]\right)^2  \right) \\
& \eqsp \eqsp \eqsp - \left( \kac_{p-1} \left[ \left(S_{p-1,\kac_{p-1}} \trans{p}\func[p]\right)^2\right] -\left(\PPsi{p-1}(\kac_{p-1})\trans{p}\func[p]\right)^2\right) \eqsp.
\end{align*}
Observing,
\begin{multline*}
\kac_{p-1} \left[ \left(S_{p-1,\kac_{p-1}} \trans{p}\func[p]\right)^2\right] -\left(\PPsi{p-1}(\kac_{p-1})\trans{p}\func[p]\right)^2 \\
= \kac_{p-1} \left( \left[S_{p-1,\kac_{p-1}}\trans{p}\func[p]-\PPsi{p-1}(\kac_{p-1})(\trans{p}\func[p])\right]^2 \right) \geq 0 \eqsp,
\end{multline*}
allows to conclude.

	\subsection{Proof of \autoref{prop:unbiasedESS}}
	\label{subsec:ProofUnbiasedESS}
	Using \eqref{eq:approxDefESS}, \eqref{eq:defGESS}, \eqref{eq:defNkacESS}, and for $\Nfunc[n]$ such that $\Nfunc[n](\epart{n}) = \left(\sum_{i=1}^{N_1} w_n^i\right)^{-1}\sum_{i=1}^{N_1} w_n^i\func[n]\left(\epart[i]{n}\right) = \omeas^{N_1}\func[n](\epart{n}) = \Pkac_n(\func[n])$, we get
\begin{align*}
\Nunkac_n(\Nfunc[n]) \eqdef \esp{\Nfunc[n](\epart{n})~\prod_{0\leq p<n}\Npot[p](\epart{p})}
&= \esp{\Pkac_n(\func[n])~\prod_{0\leq p<n}\Pkac_p(\pot[p])}.
\end{align*}
By \eqref{eq:approxDefESS-gamma}, it suffices to prove that $\esp{\unkac_n^{N_1}(\func[n])}=\unkac_{n}(\func[n]).$
We define by $\mathcal{F}^{N_1}_n$ the increasing filtration associated to the particle evolution
$\mathcal{F}^{N_1}_n \eqdef \sigma \left( \epart{p}, 0 \leq p \leq n \right) \eqsp.$
We will show that for any $p>0$ and $\func[p] \in \banach{p}$, we have
$\cesp{\Pkac_p(\func[p])}{\mathcal{F}^{N_1}_{p-1}} = \Pkac_{p-1}(\semi_p \func[p]) / \Pkac_{p-1}(\pot[p-1]) \eqsp,$
where $\semi_p$ is defined in \eqref{eq:defQ}. Indeed, by the definitions \eqref{eq:defNtransESS} of $\trans{p}$ and \eqref{eq:approxDefESS} of $\Pkac[p]$,
\begin{align*}
& \cesp{\Pkac_p(\func[p])}{\mathcal{F}^{N_1}_{p-1}} = \sum_{i=1}^{N_1} \frac{\eweight[i]{p}}{\sum_{j=1}^{N_1} \eweight[j]{p}} \cesp{\func[p](\epart[i]{p})}{\epart{p-1}} \\
& = \one_{\Theta_{p-1,\alpha}}(\epart{p-1}) \left[\dfrac{\sum_{i=1}^{N_1} \eweight[i]{p-1}\pot[p-1](\epart[i]{p-1})\trans{p}\func[p](\epart[i]{p-1})}{\sum_{i=1}^{N_1} \eweight[i]{p-1}\pot[p-1](\epart[i]{p-1})} \right] \\
& \quad + \one_{\Theta_{p-1,\alpha}^\mathrm{C}}(\epart{p-1}) \left[ \dfrac{1}{N_1}\sum_{i=1}^{N_1} \dfrac{\sum_{j=1}^{N_1} \eweight[j]{p-1}\pot[p-1](\epart[j]{p-1})\trans{p}\func[p](\epart[j]{p-1})}{\sum_{j=1}^{N_1} \eweight[j]{p-1}\pot[p-1](\epart[j]{p-1})} \right] \\
&= \dfrac{\Pkac_{p-1}(\semi_p \func[p])}{\Pkac_{p-1}(\pot[p-1])} \eqsp.
\end{align*}
The proof follows exactly along the same lines as \autoref{theo:unbiased}.
By iterating this step we get
\begin{align*}
\esp{\Punkac_n(\func[n])} &= \esp{\Pkac_{0}(\semi_1 \cdots \semi_n \func[n])}
= \esp{\semi_1 \cdots \semi_n\func[n](\epart[1]{0})} \\
&= \unkac_0\semi_1 \cdots \semi_n\func[n]
= \unkac_n(\func[n]) \eqsp.
\end{align*}
As the reader may have noticed, this unbias property doesn't depend on the definition of the sets $\Theta_{p,\alpha}$ defining the resampling times. From this observation, we underline that \autoref{prop:unbiasedESS} is also true for more general classes of resampling time criterion.

\section{Acknowledgement}
This work is supported by the Agence Nationale de la Recherche through the 2009-2012 project Big MC. The work of Christelle Vergé is financially supported by CNES (Centre National d'Etudes Spatiales) and Onera, The French Aerospace Lab.

\bibliographystyle{plain}
\bibliography{bib}
\end{document}